\documentclass[11pt,a4paper]{article}
\usepackage[utf8]{inputenc}
\usepackage{amsmath,amssymb,amsthm}
\usepackage{geometry}
\usepackage{hyperref}
\usepackage{tikz-cd}

\newtheorem{theorem}{Theorem}[section]
\newtheorem{definition}[theorem]{Definition}
\newtheorem{proposition}[theorem]{Proposition}
\newtheorem{lemma}[theorem]{Lemma}
\newtheorem{remark}[theorem]{Remark}
\newtheorem{corollary}[theorem]{Corollary}
\newtheorem{hypothesis}[theorem]{Hypothesis}

\title{\textbf{Linear Inversion and Radon Measure Representations for Infinite-Width Neural Networks with Singular Activations}}
\author{Mathias Dus \thanks{IRMA, 7 Rue René Descartes, 67000 Strasbourg, France (\texttt{mathias.dus@math.unistra.fr}).}}
\date{\today}
\everymath{\displaystyle}
\begin{document}

\maketitle
\begin{abstract}
The theoretical foundation of infinite-width shallow neural networks relies heavily on continuous integral representations and Barron spaces. Recently, harmonic analysis---specifically the Radon and Ridgelet transforms---has emerged as a powerful framework to invert these representations and reconstruct optimal network weights. However, a fundamental analytical difficulty remains: standard activation functions exhibit severe spectral singularities at the frequency origin. To bypass this divergence, existing frameworks either restrict the theory to specific activation families such as ReLU, or resort to Lizorkin distribution spaces, which require working in quotient topologies modulo polynomial subspaces.

In this paper, we bypass quotient topologies entirely. By regularizing the activation spectrum via $\widehat{g}(\rho) = (i\rho)^\alpha \widehat{\sigma}(\rho)$, we absorb origin divergences directly within standard functional spaces exhibiting asymptotic spatial decay. While this decay condition naturally excludes target functions with polynomial growth, our approach avoids quotient algebra and provides a direct reconstruction in $\mathcal{S}'(\mathbb{R}^d)$ and weighted Lebesgue spaces $L^1_p(\mathbb{R}^d)$. Building upon this framework, we formalize the parameter space using weighted Radon measures $\mathcal{M}_p(\mathbb{S}^{d-1} \times \mathbb{R})$. For activations whose regularized spectrum exhibits even parity (such as ReLU and Leaky ReLU), we prove the existence of an exact linear isometry between Barron functions and their optimal weight measures within the class of representations possessing absolutely continuous angular marginals.
\end{abstract}


\section{Introduction}
In the mathematical theory of deep learning, analyzing the expressive power and approximation capabilities of infinite-width shallow neural networks has increasingly relied on continuous integral representations.
In this setting, the discrete sum of hidden units is replaced by a Radon measure over the parameter space, shifting the mathematical focus toward characterizing the induced functional spaces and synthesizing optimal measure representations for target functions.
Building on Barron's seminal work on spectral approximation \cite{barron1993universal}, modern developments have formalized this perspective through the framework of Barron spaces and variational norms for two-layer networks \cite{e2019priori}.

To formalize these target functions, much of the foundational literature relies on the concept of spectral Barron spaces, which characterize network representability through decay properties of the target's Fourier transform. Originally introduced by Barron for sigmoidal activations via a first-order spectral moment condition $\int_{\mathbb{R}^d} |\xi| |\widehat{f}(\xi)| \, d\xi < \infty$ \cite{barron1993universal}, this perspective was subsequently extended to second-order moments $\int_{\mathbb{R}^d} (1+|\xi|^2) |\widehat{f}(\xi)| \, d\xi < \infty$ to accommodate homogeneous activations such as the Rectified Linear Unit (ReLU) \cite{klukowskiBarron2018, e2019priori}, and more generally to higher-order approximation rates for shallow neural networks \cite{siegel2020approximation}. This spectral perspective is highly advantageous: it translates the spatial complexity of the target function into an integrability condition in the frequency domain, creating a natural theoretical bridge toward harmonic analysis and measure-theoretic representations. 

The structural importance of spectral Barron spaces is further underscored by their ability to mitigate the curse of dimensionality \cite{barron1993universal, bach2017breaking}. In classical linear approximation theory, achieving a dimension-independent convergence rate of $\mathcal{O}(m^{-1/2})$ typically requires an order of Sobolev regularity proportional to half the ambient dimension ($s \ge d/2$). In contrast, functions in the Barron space achieve this $\mathcal{O}(m^{-1/2})$ Monte-Carlo rate---where $m$ denotes the width of the shallow network---without requiring such prohibitive spatial differentiation. As recently re-examined by Schavemaker \cite{schavemaker2025barron}, this behavior stems from an atypical, non-classical structural regularity intrinsic to the spectral decay condition. Consequently, Barron spaces provide an effective functional framework for high-dimensional approximation, making the constructive synthesis of their representing measures a central analytical objective.

Having established the theoretical robustness of Barron spaces, the subsequent operational question shifts to the computation of these representing measures. Recent literature has successfully leveraged tools from harmonic analysis, establishing a direct connection between continuous neural networks, the Radon transform, and Ridgelet theory \cite{radon1917uber, candes1998ridgelets, candes1999harmonic}.
Ongie et al. \cite{ongie2020function} characterized the function space induced by bounded-norm infinite-width ReLU networks, relating the corresponding representation norm to the Radon transform of a fractional Laplacian of the target function. Building on this functional perspective, Parhi and Nowak \cite{parhi2021banach} formulated neural network learning as a continuous-domain inverse problem with Radon-domain regularization and established representer theorems connecting its solutions to ridge splines. Related variational formulations further provide a unified perspective connecting kernel methods, neural networks, Radon-domain regularization, and sparse ridge-function representations \cite{unser2024kernel}. In a complementary direction, Sonoda and Murata \cite{sonoda2017neural} established universal approximation results for neural networks with unbounded activation functions, including ReLU, using the Ridgelet transform and its interpretation as a backprojection filter in the Radon domain.

Despite these significant advancements, the harmonic analysis of neural networks remains structurally constrained by the handling of spectral singularities.
The inversion of the Radon transform involves frequency-domain filtering by powers of the radial frequency variable, which may exhibit a non-integrable singularity at the origin \(\rho=0\).
Different frameworks address this analytical difficulty in different ways. For ReLU networks, Savarese et al. \cite{savarese2019how} characterized the univariate infinite-width bounded-norm function space, while Ongie et al. \cite{ongie2020function} extended this characterization to the multivariate setting. In the more general setting of unbounded activations, Sonoda and Murata \cite{sonoda2017neural} instead formulate the ridgelet transform in the framework of Lizorkin distributions, where the problematic polynomial components are factored out. While mathematically rigorous, working within Lizorkin spaces requires operating in quotient topologies modulo polynomial subspaces by testing against functions with vanishing moments of all orders.

Rather than resorting to quotient spaces, a natural alternative is to impose sufficient spatial decay on the target and work within standard weighted functional spaces, such as the weighted Lebesgue spaces \(L^1_p(\mathbb{R}^d)\).
By imposing sufficient spatial decay on the target, non-decaying polynomial components are excluded from the admissible function class, allowing the representation to be formulated directly in the underlying function space rather than modulo polynomial components.
This perspective retains the standard functional framework on \(\mathbb{R}^d\), represents the target directly rather than through equivalence classes modulo polynomials, and provides a natural bridge to Banach-space duality and measure-theoretic representations.

\paragraph{Contributions.}
In this paper, we present a rigorous distributional framework for the generalized Radon transform that circumvents Lizorkin quotient spaces and provides an explicit synthesis of network weights. Our main contributions form a systematic three-step analytical progression: 
\begin{enumerate}
    \item \textbf{A Direct Tempered Distribution Framework for the Radon Transform:} As a foundational step, we define the generalized Radon transform \(\mathcal{R}_\sigma\) for tempered distributions \(\sigma \in \mathcal{S}'(\mathbb{R})\) using the Hörmander pullback along affine submersions. This formulates the continuous neural network representation directly within the standard topology of \(\mathcal{S}'(\mathbb{R}^d)\). We further address the low-frequency singularity by introducing a generalized inversion filter based on the regularized spectrum \(\widehat{g}(\rho) = (i\rho)^\alpha \widehat{\sigma}(\rho)\), which provides a well-defined distributional reconstruction under the stated decay and regularity assumptions.

\item \textbf{Transition to Weighted Radon Measures and Angular Absolute Continuity:} Moving beyond the purely distributional setting, we formalize the network's parameter space using the Banach space of weighted Radon measures $\mathcal{M}_p(\mathbb{S}^{d-1} \times \mathbb{R})$. Within this dual space, we prove a geometric disintegration theorem establishing that our canonical filtered analysis operator yields measures whose spherical marginals are strictly absolutely continuous with respect to the normalized spherical Lebesgue measure $\underline{dw}$. Whenever the dimension satisfies $\alpha + d - 2 \ge k_\sigma + 2$ (which reduces to $d \ge k_\sigma + 2$ for standard second-order activations such as ReLU where $\alpha = 2$), this provides a topological representation theorem for target functions in $L^1_p(\mathbb{R}^d)$ without polynomial quotienting.
    
  \item \textbf{Canonical Linear Isometric Selection of Optimal Barron Representations:} In the classical measure-based formulation, the Barron norm is defined through an infimum over Radon measures representing a given target function, and norm-minimizing representations need not be unique. In contrast, we consider a generalized Barron space $\mathcal{B}$ obtained by restricting the admissible representing measures to those whose angular marginals are absolutely continuous with respect to $\underline{dw}$. While our representation framework also accommodates activations with non-symmetric regularized spectra, such as the Exponential Linear Unit (ELU), we prove that, whenever the regularized spectrum $\widehat{g}$ is even—as is the case for ReLU, Leaky ReLU, the Absolute Value activation, the Heaviside step function, and polynomial B-splines—a canonical norm-minimizing representation can be selected through an operator $P:\mathcal{B}\to\mathcal{M}_p$ that is linear and isometric.

This structural result has several analytical and constructive consequences:

    \begin{itemize}
        \item \textit{Closed-form synthesis without variational optimization:} Rather than solving an infinite-dimensional variational optimization problem to determine optimal network weights, the minimal-norm measure $\mu_{\mathrm{opt}} = P(f)$ is given explicitly by a closed-form frequency filtering formula.
        \item \textit{Direct evaluation of representational complexity:} The quotient Barron norm is directly accessible as $\|f\|_{\mathcal{B}} = \|Pf\|_{\mathrm{TV}, p}$, bypassing the infimum computation while strictly preserving the superposition principle $P(f_1 + f_2) = P(f_1) + P(f_2)$.
        \item \textit{Sampling distribution for finite-width networks:} Beyond the continuum limit, normalizing the total variation measure $|\mu_{\mathrm{opt}}|$ yields an explicit probability distribution on the parameter cylinder from which shallow networks of finite width $m$ can be synthesized via Monte-Carlo sampling to achieve the dimension-independent $\mathcal{O}(m^{-1/2})$ approximation rate.
    \end{itemize}
\end{enumerate}

\paragraph{Organization of the Paper.}
The remainder of this paper is structured as follows. Section \ref{sec:generalized_radon} constructs the generalized Radon transform $\mathcal{R}_\sigma$ using the Hörmander pullback for distributional activations. Section \ref{sec:adjoint_radon} determines the dual synthesis operator $\mathcal{R}_\sigma^*$ through weak duality. Section \ref{sec:caracterization_barron_distrib} forms the analytical core of our distributional framework by introducing the regularized spectrum $\widehat{g}$ and the inversion filter $\mathcal{T}_g$, culminating in the exact neural representation theorem in the sense of distributions. Section \ref{sec:caracterization_barron_tv} bridges this harmonic framework with the Banach space formulation by introducing weighted Radon measures $\mathcal{M}_p$, establishing the fiber disintegration lemma, and proving the Sobolev--spectral embedding. Section \ref{sec:optimal_operator} investigates the generalized Barron space $\mathcal{B}$, analyzes the central inversion symmetry of the parameter cylinder, and constructs the optimal linear isometric operator $P$ for activations with even regularized spectra. Finally, Appendix \ref{sec:appendix_distributions} compiles essential distribution-theoretic theorems regarding parameter-dependent distribution integrals, duality on weighted spaces, and regularized spectral products used throughout the proofs.

\subsection{Notations and Conventions}
We work in the Euclidean space $\mathbb{R}^d$ equipped with its standard inner product $\langle \cdot, \cdot \rangle$ and the associated norm $|\cdot|$.
The unit cylinder $\mathbb{S}^{d-1} \times \mathbb{R}$ parameterizes the set of oriented hyperplanes, where $\mathbb{S}^{d-1} = \{w \in \mathbb{R}^d \mid |w| = 1\}$ denotes the unit sphere equipped with its normalized Lebesgue measure $\underline{dw}$ (such that $\int_{\mathbb{S}^{d-1}} \underline{dw} = 1$).
For spectral processing, we adopt the usual convention for the Fourier transform $\mathcal{F}$ of a function $u \in L^1(\mathbb{R}^n)$ (with $n=1$ or $n=d$):
\begin{equation*}
    \mathcal{F}\{u\}(\xi) = \widehat{u}(\xi) = \int_{\mathbb{R}^n} u(x) e^{-i \langle \xi, x \rangle} \, dx.
\end{equation*}
The associated inversion formula is given by:
\begin{equation*}
    \mathcal{F}^{-1}\{v\}(x) = \frac{1}{(2\pi)^n} \int_{\mathbb{R}^n} v(\xi) e^{i \langle x, \xi \rangle} \, d\xi.
\end{equation*}
These definitions extend to tempered distributions by the standard duality $\langle \widehat{T}, \phi \rangle_{\mathbb{R}^n} = \langle T, \widehat{\phi} \rangle_{\mathbb{R}^n}$, for $T \in \mathcal{S}'(\mathbb{R}^n)$ and $\phi \in \mathcal{S}(\mathbb{R}^n)$.

For any polynomial growth order $p \ge 0$, we define the scalar polynomial weight function $W_p(b) := (1 + |b|)^p$ on $\mathbb{R}$. The spatially weighted Lebesgue space $L^1_p(\mathbb{R}^d) := L^1(\mathbb{R}^d; (1+|x|)^p dx)$ is equipped with the norm $\|f\|_{L^1_p} = \int_{\mathbb{R}^d} (1+|x|)^p |f(x)| \, dx$. Throughout this paper, continuous activation functions $\sigma \in C^0(\mathbb{R})$ will be characterized by their polynomial growth order $k_\sigma \ge 0$, satisfying $|\sigma(t)| \le C_\sigma (1 + |t|)^{k_\sigma}$ for all $t \in \mathbb{R}$.
\section{Definition of the Generalized Radon Transform}\label{sec:generalized_radon}

In order to manipulate standard activation functions (such as ReLU, the sigmoid, or the sign function) that are not integrable on $\mathbb{R}$,  we place ourselves within the abstract framework of Schwartz's theory of distributions by utilizing the concept of the pullback of a distribution.

\begin{definition}[Activation Distribution]\label{def:sig_distribution}
The activation $\sigma$ is assumed to be a real-valued tempered distribution on the real line, i.e., $\sigma \in \mathcal{S}'_\mathbb{R}(\mathbb{R})$ (satisfying $\langle \sigma, \overline{\phi} \rangle = \overline{\langle \sigma, \phi \rangle}$ for all test functions $\phi \in \mathcal{S}(\mathbb{R})$).
\end{definition}

To proceed, we need the concept of the pullback of a distribution.
\begin{definition}[Pullback along Affine Slices on $\mathcal{S}'$]
Let $w \in \mathbb{S}^{d-1}$ and $b \in \mathbb{R}$, and let $\phi_{w,b} : \mathbb{R}^d \to \mathbb{R}$ be the affine map $\phi_{w,b}(x) = \langle w, x \rangle - b$. For any test function $\psi \in \mathcal{S}(\mathbb{R}^d)$, the fiber integral
\begin{equation}
    \psi_{\phi_{w,b}}(t) := \int_{\langle w, x \rangle = t+b} \psi(x) \, d\mathcal{H}^{d-1}(x)
\end{equation}
defines a continuous linear mapping from $\mathcal{S}(\mathbb{R}^d)$ into $\mathcal{S}(\mathbb{R})$, where $d\mathcal{H}^{d-1}$ denotes the $(d-1)$-dimensional Euclidean Lebesgue measure on the affine hyperplane $\{x \in \mathbb{R}^d \mid \langle w, x \rangle = t+b\}$.

By duality, for any tempered distribution $T \in \mathcal{S}'(\mathbb{R})$, the pullback distribution $\phi_{w,b}^* T \in \mathcal{S}'(\mathbb{R}^d)$ is uniquely defined by:
\begin{equation}
    \langle \phi_{w,b}^* T, \psi \rangle_{\mathbb{R}^d} = \langle T, \psi_{\phi_{w,b}} \rangle_{\mathbb{R}} \quad \forall \psi \in \mathcal{S}(\mathbb{R}^d).
\end{equation}
\end{definition}

Applying this construction to an activation distribution $\sigma \in \mathcal{S}'(\mathbb{R})$ yields a rigorously defined tempered distribution $\phi_{w,b}^* \sigma \in \mathcal{S}'(\mathbb{R}^d)$, enabling the definition of the generalized Radon transform without requiring pointwise convergent spatial integrals.

\begin{definition}[Generalized Radon Transform $\mathcal{R}_\sigma$]
For an analysis function $f \in \mathcal{S}(\mathbb{R}^d)$ and an activation distribution $\sigma \in \mathcal{S}'(\mathbb{R})$, the generalized Radon transform, denoted $\mathcal{R}_\sigma f$, is defined on the parameter cylinder $\mathbb{S}^{d-1} \times \mathbb{R}$ by the distributional duality bracket on $\mathbb{R}^d$:
\begin{equation}
    \mathcal{R}_\sigma f(w, b) = \langle \phi_{w,b}^* \sigma, f \rangle_{\mathbb{R}^d},
\end{equation}
where $\phi_{w,b}^* \sigma$ is the pullback distribution of $\sigma$ by the submersion $x \mapsto \langle w, x \rangle - b$.
\end{definition}

\begin{remark}\label{rm:espace_arrive_R_sigma}
This abstract definition is a generalization of the formal spatial integral $\int_{\mathbb{R}^d} f(x) \sigma(\langle w, x \rangle - b) \, dx$. As shown in the proof below, for all $w \in \mathbb{S}^{d-1}$, the resulting mapping $b \mapsto \mathcal{R}_\sigma f(w, b)$ is infinitely differentiable. Furthermore, $\mathcal{R}_\sigma f(w, \cdot)$ possesses at most slow (polynomial) growth with respect to the bias variable $b$.
\end{remark}

\begin{proof}[Proof of the properties stated in Remark \ref{rm:espace_arrive_R_sigma}]
Let $f \in \mathcal{S}(\mathbb{R}^d)$ and an orientation $w \in \mathbb{S}^{d-1}$ be fixed. We study the function of the bias variable defined by the pullback:
\begin{equation}
    h(b) := \mathcal{R}_\sigma f(w, b) = \langle \phi_{w,b}^* \sigma, f \rangle_{\mathbb{R}^d}.
\end{equation}

Let us define the reduced 1D test function obtained by integrating over these spatial slices:
\begin{equation}
    \psi_w(s) := \int_{\langle w, x \rangle = s} f(x) \, d\mathcal{H}^{d-1}(x).
\end{equation}
Since $f \in \mathcal{S}(\mathbb{R}^d)$, we have $\psi_w \in \mathcal{S}(\mathbb{R})$.

The action of the spatial pullback is strictly equivalent to the 1D duality bracket evaluated at the shifted coordinate $s = t+b$:
\begin{equation}\label{eq:Rsigmaf_form}
    h(b) = \langle \sigma(t), \psi_w(t + b) \rangle_t.
\end{equation}

Now, according to the structure theorem \ref{th:structure_distrib}, since $\sigma \in \mathcal{S}'(\mathbb{R})$, there exist an integer $m \ge 0$ and a continuous function $F$ with at most polynomial growth of order $k$ (i.e., $|F(t)| \le C(1+|t|)^k$) such that $\sigma = D^m F$ in the sense of distributions.
By definition of the distributional derivative, the action shifts to the test function:
\begin{equation} \label{eq:h_IPP}
    h(b) = \langle D^m F(t), \psi_w(t + b) \rangle_t = (-1)^m \int_{\mathbb{R}} F(t) \frac{\partial^m}{\partial t^m} \psi_w(t + b) \, dt.
\end{equation}

As $F$ has at most polynomial growth and $\psi_w$ is in $\mathcal{S}(\mathbb{R})$, we can apply the classical theorem of differentiation under the integral sign to prove that $h \in C^\infty(\mathbb{R})$. Moreover, since the integral in \eqref{eq:h_IPP} is absolutely convergent, we can perform a change of variables to obtain:

\[
h(b) = (-1)^m \int_{\mathbb{R}} F(t - b) \frac{\partial^m}{\partial t^m} \psi_w(t) \, dt,
\]
so that using the polynomial growth of $F$ (of order $k \in \mathbb{N}$) given by the structure theorem \ref{th:structure_distrib}:

\[
|h(b)| \leq C \int_{\mathbb{R}} (1 + |t|^k + |b|^k)  \left| \frac{\partial^m}{\partial t^m} \psi_w(t) \right| \, dt,
\]
which immediately gives the polynomial growth of $h$ and concludes the proof.
\end{proof}

The Fourier transform in the bias variable $b$ of $\mathcal{R}_\sigma f(w, \cdot)$ plays a fundamental role throughout our analysis, and is characterized as follows.
\begin{proposition}[Spectral Representation of the Generalized Radon Transform]\label{prop:spectral_representation}
Let $f \in \mathcal{S}(\mathbb{R}^d)$ and $\sigma \in \mathcal{S}'(\mathbb{R})$ be an activation distribution. For any fixed orientation $w \in \mathbb{S}^{d-1}$, the 1D Fourier transform of the generalized Radon transform with respect to the bias variable $b$ is given by:
\begin{equation}\label{eq:tfb_Rf}
    \mathcal{F}_b\{\mathcal{R}_\sigma f(w, \cdot)\}(\rho) = \widehat{f}(\rho w) \overline{\widehat{\sigma}(\rho)}
\end{equation}
in the sense of tempered distributions on $\mathbb{R}$.
\end{proposition}

\begin{proof}
By definition of the Fourier transform on tempered distributions, the action of the spectrum on a test function $\phi \in \mathcal{S}(\mathbb{R})$ shifts to the spatial domain:
\begin{equation*}
    \langle \mathcal{F}_b\{\mathcal{R}_\sigma f(w, \cdot)\}, \phi \rangle_{\mathbb{R}} = \langle \mathcal{R}_\sigma f(w, \cdot), \widehat{\phi} \rangle_{\mathbb{R}}.
\end{equation*}

According to \eqref{eq:Rsigmaf_form}:
\begin{equation*}
    \mathcal{R}_\sigma f(w, b) = \langle \sigma(t), \psi_w(t+b) \rangle_t,
\end{equation*}
where $\psi_w(s) = \int_{\langle w, x \rangle = s} f(x) \, d\mathcal{H}^{d-1}(x) \in \mathcal{S}(\mathbb{R})$. Since $\mathcal{R}_\sigma f(w, \cdot)$ is regular and has at most polynomial growth, its outer action on $\widehat{\phi} \in \mathcal{S}(\mathbb{R})$ is given by a regular Lebesgue integral over $b$:
\begin{equation*}
    \langle \mathcal{R}_\sigma f(w, \cdot), \widehat{\phi} \rangle_{\mathbb{R}} = \int_{\mathbb{R}} \langle \sigma(t), \psi_w(t+b) \rangle_t \widehat{\phi}(b) \, db.
\end{equation*}

 To rigorously justify that the mapping $t \mapsto \psi_w(t+b)\widehat{\phi}(b)$ belongs to $\mathcal{S}(\mathbb{R})$ uniformly in $b$, we must bound its Schwartz semi-norms independently of the shift parameter. For any integers $k, m \ge 0$, the corresponding semi-norm evaluates to $\sup_{t \in \mathbb{R}} |t^k \psi_w^{(m)}(t+b) \widehat{\phi}(b)|$. Applying the elementary inequality $|t|^k = |(t+b) - b|^k \le 2^k (|t+b|^k + |b|^k)$, we obtain the pointwise bound:
\begin{equation*}
    |t^k \psi_w^{(m)}(t+b) \widehat{\phi}(b)| \le 2^k \left( |t+b|^k |\psi_w^{(m)}(t+b)| |\widehat{\phi}(b)| + |\psi_w^{(m)}(t+b)| |b|^k |\widehat{\phi}(b)| \right).
\end{equation*}
Because $\psi_w$ resides in $\mathcal{S}(\mathbb{R})$, its spatial growth and derivatives are globally bounded. Specifically, there exist finite constants $M_1, M_2 > 0$ such that $|s|^k |\psi_w^{(m)}(s)| \le M_1$ and $|\psi_w^{(m)}(s)| \le M_2$ for all $s \in \mathbb{R}$. Consequently, the semi-norm is bounded pointwise by $2^k (M_1 |\widehat{\phi}(b)| + M_2 |b|^k |\widehat{\phi}(b)|)$. Since the test function $\widehat{\phi}$ also belongs to $\mathcal{S}(\mathbb{R})$, this bounding function has rapid decay and is strictly integrable over $\mathbb{R}$. 

The theorem of integration of distributions depending on a parameter (Theorem \ref{th:inversion_integrale_crochet}) allows us to strictly commute the integral and the duality bracket:
\begin{equation*}
    \int_{\mathbb{R}} \langle \sigma(t), \psi_w(t+b) \rangle_t \widehat{\phi}(b) \, db = \left\langle \sigma(t), \int_{\mathbb{R}} \psi_w(t+b) \widehat{\phi}(b) \, db \right\rangle_t.
\end{equation*}

To analyze the inner test function, let us perform the change of variables $s = t+b$ (thus $db = ds$):
\begin{equation*}
    u(t) := \int_{\mathbb{R}} \psi_w(s) \widehat{\phi}(s-t) \, ds.
\end{equation*}
By expanding the Fourier transform $\widehat{\phi}(s-t) = \int_{\mathbb{R}} \phi(\rho) e^{-i \rho (s-t)} \, d\rho$ and applying standard Fubini's theorem on $\mathcal{S}(\mathbb{R})$, we reformulate $u(t)$:
\begin{align*}
    u(t) &= \int_{\mathbb{R}} \psi_w(s) \left( \int_{\mathbb{R}} \phi(\rho) e^{-i \rho s} e^{i \rho t} \, d\rho \right) ds \\
    &= \int_{\mathbb{R}} \phi(\rho) \left( \int_{\mathbb{R}} \psi_w(s) e^{-i \rho s} \, ds \right) e^{i \rho t} \, d\rho.
\end{align*}

The inner integral over $s$ is exactly the 1D Fourier transform of $\psi_w$. By the classical slice integration theorem, this coincides with the spectrum of the target function evaluated radially: $\widehat{\psi_w}(\rho) = \widehat{f}(\rho w)$. Thus:
\begin{equation*}
    u(t) = \int_{\mathbb{R}} \phi(\rho) \widehat{f}(\rho w) e^{i \rho t} \, d\rho.
\end{equation*}

To identify this expression as a standard forward Fourier transform $\widehat{\eta}(t) = \int_{\mathbb{R}} \eta(v) e^{-i v t} \, dv$, we perform the change of variables $v = -\rho$. Noting that $d\rho = -dv$ and that reversing the integration limits cancels the minus sign, we obtain:
\begin{equation*}
    u(t) = \int_{\mathbb{R}} \phi(-v) \widehat{f}(-v w) e^{-i v t} \, dv = \widehat{\eta}(t),
\end{equation*}
where the constructed test function is $\eta(v) = \phi(-v) \widehat{f}(-v w) \in \mathcal{S}(\mathbb{R})$. We can now apply the strict duality definition of the Fourier transform on distributions:
\begin{equation*}
    \langle \sigma(t), u(t) \rangle_t = \langle \sigma(t), \widehat{\eta}(t)  \rangle_t = \langle \widehat{\sigma}(v), \phi(-v) \widehat{f}(-v w) \rangle_v.
\end{equation*}

By the reflection property of distributions, $\langle \widehat{\sigma}(v), \zeta(-v) \rangle_v = \langle \widehat{\sigma}(-\rho), \zeta(\rho) \rangle_\rho$ for any test function $\zeta \in \mathcal{S}(\mathbb{R})$. Furthermore, since $\sigma$ is a real-valued distribution, its spectrum satisfies the Hermitian symmetry $\widehat{\sigma}(-\rho) = \overline{\widehat{\sigma}(\rho)}$ in $\mathcal{S}'(\mathbb{R})$. Setting $\zeta(\rho) = \widehat{f}(\rho w) \phi(\rho)$, we obtain:
\begin{equation*}
    \langle \widehat{\sigma}(-\rho), \phi(\rho) \widehat{f}(\rho w) \rangle_\rho = \left\langle \overline{\widehat{\sigma}(\rho)} \widehat{f}(\rho w), \phi(\rho) \right\rangle_\rho.
\end{equation*}

Since the duality identity:
\begin{equation*}
    \langle \mathcal{F}_b\{\mathcal{R}_\sigma f(w, \cdot)\}, \phi \rangle_{\mathbb{R}} = \left\langle \widehat{f}(\rho w) \overline{\widehat{\sigma}(\rho)}, \phi \right\rangle_{\mathbb{R}}
\end{equation*}
holds for every test function $\phi \in \mathcal{S}(\mathbb{R})$, we conclude that:
\begin{equation}
    \mathcal{F}_b\{\mathcal{R}_\sigma f(w, \cdot)\}(\rho) = \widehat{f}(\rho w) \overline{\widehat{\sigma}(\rho)} \quad \text{in } \mathcal{S}'(\mathbb{R}).
\end{equation}
\end{proof}
\section{Determination of the Dual (Adjoint) Operator by Duality}\label{sec:adjoint_radon}
 We now construct the adjoint operator of $\mathcal{R}_\sigma$ rigorously. 
\begin{proposition}[Dual Synthesis Operator $\mathcal{R}_\sigma^*$]
The adjoint operator $\mathcal{R}_\sigma^* : \mathcal{S}(\mathbb{S}^{d-1} \times \mathbb{R}) \to \mathcal{S}'(\mathbb{R}^d)$, defined abstractly by the duality identity:
\begin{equation}\label{eq:def_duality_R_sigma}
    \langle \mathcal{R}_\sigma f, h \rangle_{\mathbb{S}^{d-1} \times \mathbb{R}} = \langle \mathcal{R}_\sigma^* h, f\rangle_{\mathbb{R}^d} \quad \forall f \in \mathcal{S}(\mathbb{R}^d), \forall h \in \mathcal{S}(\mathbb{S}^{d-1} \times \mathbb{R}),
\end{equation}
is given point-wise by the $C^\infty$ function with polynomial growth:
\begin{equation}\label{eq:def_Rstarsigma}
    \mathcal{R}_\sigma^* h(x) = \int_{\mathbb{S}^{d-1}} \langle \sigma(t), h(w, \langle w, x \rangle - t) \rangle_t \, \underline{dw}.
\end{equation}
\end{proposition}

\begin{proof}
\noindent\textbf{Step 1: Smooth regularity and polynomial growth.}\par
According to the structure theorem for tempered distributions (Theorem \ref{th:structure_distrib}), there exist an integer $m \ge 0$, an integer $k \ge 0$, and a continuous function $F \in C^0(\mathbb{R})$ satisfying $|F(t)| \le C_F (1+|t|)^k$ such that $\sigma = D^m F$ in $\mathcal{S}'(\mathbb{R})$.
For any fixed $x \in \mathbb{R}^d$ and $w \in \mathbb{S}^{d-1}$, the slice function $t \mapsto h(w, \langle w, x \rangle - t)$ belongs to $\mathcal{S}(\mathbb{R})$. Shifting the distributional derivative onto the test function yields:
\begin{equation*}
    \langle \sigma(t), h(w, \langle w, x \rangle - t) \rangle_t = (-1)^m \int_{\mathbb{R}} F(t) \frac{\partial^m}{\partial t^m} \big[ h(w, \langle w, x \rangle - t) \big] \, dt = \int_{\mathbb{R}} F(t) (\partial_s^m h)(w, \langle w, x \rangle - t) \, dt,
\end{equation*}
where the alternating signs cancel because $\frac{\partial}{\partial t} [h(w, \langle w, x \rangle - t)] = -(\partial_s h)(w, \langle w, x \rangle - t)$. Hence:
\begin{equation*}
    \mathcal{R}_\sigma^* h(x) = \int_{\mathbb{S}^{d-1}} \left( \int_{\mathbb{R}} F(t) (\partial_s^m h)(w, \langle w, x \rangle - t) \, dt \right) \underline{dw}.
\end{equation*}
For any multi-index $\alpha \in \mathbb{N}^d$, differentiating the integrand with respect to $x$ brings a factor $w^\alpha (\partial_s^{m+|\alpha|} h)(w, \langle w, x \rangle - t)$. For $x$ in any compact ball $B(0, R)$, since $|\langle w, x \rangle| \le R$, we have $|t| \le |t - \langle w, x \rangle| + R$. The rapid decay of $h \in \mathcal{S}(\mathbb{S}^{d-1} \times \mathbb{R})$ ensures that the derivative is uniformly dominated by an $L^1(\mathbb{R}, dt)$ function independent of $w$. Leibniz's differentiation rule under both integrals proves that $\mathcal{R}_\sigma^* h \in C^\infty(\mathbb{R}^d)$. Furthermore, substituting $s = \langle w, x \rangle - t$ and using the elementary inequality $|F(\langle w, x \rangle - s)| \le C_F 2^k (1 + |x|^k + |s|^k)$ yields:
\begin{equation*}
    |D_x^\alpha \mathcal{R}_\sigma^* h(x)| \le C_F 2^k (1 + |x|^k) \sup_{w \in \mathbb{S}^{d-1}} \int_{\mathbb{R}} (1+|s|^k) |(\partial_s^{m+|\alpha|} h)(w, s)| \, ds \le M_\alpha (1+|x|)^k,
\end{equation*}
which proves that $\mathcal{R}_\sigma^* h$ and all its derivatives have at most polynomial growth, hence defining a tempered distribution in $\mathcal{S}'(\mathbb{R}^d)$.

\medskip
\noindent\textbf{Step 2: Duality identity.}\par
Next, we prove that the operator defined by \eqref{eq:def_Rstarsigma} satisfies the duality property \eqref{eq:def_duality_R_sigma}. To do so, let $f \in \mathcal{S}(\mathbb{R}^d)$ and $h \in \mathcal{S}(\mathbb{S}^{d-1} \times \mathbb{R})$ be a test function. As $\mathcal{R}_\sigma^* h$ is a $C^\infty$ function with at most polynomial growth, we can replace the duality brackets by integrals:
\begin{equation*}
    \langle \mathcal{R}_\sigma^* h, f \rangle_{\mathbb{R}^d} = \int_{\mathbb{R}^d} f(x) \left( \int_{\mathbb{S}^{d-1}} \langle \sigma(t), h(w, \langle w, x \rangle - t) \rangle_t \, \underline{dw} \right) dx.
\end{equation*}

To rigorously justify swapping the order of integration over $x$ and $w$, recall from Step 1 that the fiber bracket is uniformly bounded on the sphere: there exists $M > 0$ such that $|\langle \sigma(t), h(w, \langle w, x \rangle - t) \rangle_t| \le M(1+|x|)^k$ for all $w \in \mathbb{S}^{d-1}$. 
Because $f \in \mathcal{S}(\mathbb{R}^d)$ and $\underline{dw}$ is a probability measure, the joint integrand is absolutely integrable:
\begin{equation*}
    \int_{\mathbb{S}^{d-1}} \int_{\mathbb{R}^d} |f(x)| |\langle \sigma(t), h(w, \langle w, x \rangle - t) \rangle_t| \, dx \, \underline{dw} \le M \int_{\mathbb{R}^d} |f(x)|(1+|x|)^k \, dx < \infty.
\end{equation*}
By Fubini--Tonelli, we can legitimately swap the order of the integrals over $\mathbb{R}^d$ and $\mathbb{S}^{d-1}$:
\begin{equation*}
    \langle \mathcal{R}_\sigma^* h, f \rangle_{\mathbb{R}^d} = \int_{\mathbb{S}^{d-1}} \left( \int_{\mathbb{R}^d} f(x) \langle \sigma(t), h(w, \langle w, x \rangle - t) \rangle_t \, dx \right) \underline{dw}.
\end{equation*}

We now apply the theorem of integration of distributions depending on parameters (Theorem \ref{th:inversion_integrale_crochet}) to pull the distributional bracket outside the integral over $x$. To do so, we must prove that the function $x \mapsto f(x) h(w, \langle w, x \rangle - t)$ is Bochner-integrable with respect to the topology of $\mathcal{S}(\mathbb{R})$. Let us bound its Schwartz semi-norms in the variable $t$. For any integers $k, m \ge 0$, we evaluate:
\begin{equation*}
    N_{k,m}(x, w) := \sup_{t \in \mathbb{R}} \left| t^k \frac{\partial^m}{\partial t^m} \Big( f(x) h(w, \langle w, x \rangle - t) \Big) \right| = |f(x)| \sup_{t \in \mathbb{R}} \left| t^k (\partial_s^m h)(w, \langle w, x \rangle - t) \right|.
\end{equation*}

Using the algebraic inequality $|t|^k \le 2^k ( |t - \langle w, x \rangle|^k + |\langle w, x \rangle|^k ) \le 2^k ( |t - \langle w, x \rangle|^k + |x|^k )$, we obtain the upper bound:
\begin{equation*}
    N_{k,m}(x, w) \le 2^k |f(x)| \left( \sup_{s \in \mathbb{R}} |s^k \partial_s^m h(w, s)| + |x|^k \sup_{s \in \mathbb{R}} |\partial_s^m h(w, s)| \right).
\end{equation*}

Since $h \in \mathcal{S}(\mathbb{S}^{d-1} \times \mathbb{R})$, its partial derivatives and their polynomial weightings are globally bounded by finite constants independently of the orientation $w$. Thus, there exists a constant $C_{k,m} > 0$ such that:
\begin{equation*}
    N_{k,m}(x, w) \le C_{k,m} |f(x)| (1 + |x|^k).
\end{equation*}

Because the test function $f$ belongs to $\mathcal{S}(\mathbb{R}^d)$, the bounding function $x \mapsto C_{k,m} |f(x)| (1 + |x|^k)$ is rapidly decaying and strictly integrable over $\mathbb{R}^d$. This confirms that the integral of the semi-norms over the spatial domain is finite. Consequently, Theorem \ref{th:inversion_integrale_crochet} rigorously applies, allowing us to safely commute the spatial integral with the distributional action of $\sigma$:
\begin{equation}\label{eq:f_Rstarsigma_h}
    \langle \mathcal{R}_\sigma^* h, f \rangle_{\mathbb{R}^d} = \int_{\mathbb{S}^{d-1}} \left\langle \sigma(t), \int_{\mathbb{R}^d} f(x) h(w, \langle w, x \rangle - t) \, dx \right\rangle_t \, \underline{dw}.
\end{equation}

To connect this bracket with the spatial pullback $\phi_{w,b}^* \sigma$, we must express the inner spatial integral using integration over the affine fibers. Let us group the points $x \in \mathbb{R}^d$ by their projection along the vector $w$, by setting $s = \langle w, x \rangle$:
\begin{align*}
    \int_{\mathbb{R}^d} f(x) h(w, \langle w, x \rangle - t) \, dx &= \int_{\mathbb{R}} \left( \int_{\langle w, x \rangle = s} f(x) \, d\mathcal{H}^{d-1}(x) \right) h(w, s - t) \, ds \\
    &= \int_{\mathbb{R}} \left( \int_{\langle w, x \rangle = t+b} f(x) \, d\mathcal{H}^{d-1}(x) \right) h(w, b) \, db.
\end{align*}
Evaluating the Schwartz semi-norms of $t \mapsto \psi_w(t+b) h(w, b)$ shows that they are bounded by $C'_{k,m} |h(w, b)| (1 + |b|^k)$, which is integrable over $\mathbb{R}$ with respect to $b$. Applying Theorem \ref{th:inversion_integrale_crochet} again allows us to pull the integral over $b$ outside the bracket:
\begin{equation*}
    \left\langle \sigma(t), \int_{\mathbb{R}^d} f(x) h(w, \langle w, x \rangle - t) \, dx \right\rangle_t = \int_{\mathbb{R}} \left\langle \sigma(t), \int_{\langle w, x \rangle - b = t} f(x) \, d\mathcal{H}^{d-1}(x) \right\rangle_t h(w, b) \, db.
\end{equation*}
By definition of the pullback distribution: 
\begin{equation*}
    \left\langle \sigma(t), \int_{\langle w, x \rangle - b = t} f(x) \, d\mathcal{H}^{d-1}(x) \right\rangle_t = \langle \phi_{w,b}^* \sigma, f \rangle_{\mathbb{R}^d}.
\end{equation*}
By substituting this into \eqref{eq:f_Rstarsigma_h}:
\begin{align*}
    \langle \mathcal{R}_\sigma^* h, f \rangle_{\mathbb{R}^d} &= \int_{\mathbb{S}^{d-1}} \left( \int_{\mathbb{R}} \langle \phi_{w,b}^* \sigma, f \rangle_{\mathbb{R}^d} \, h(w, b) \, db \right) \underline{dw} \\
    &= \int_{\mathbb{S}^{d-1}} \int_{\mathbb{R}} \mathcal{R}_\sigma f(w, b) h(w, b) \, db \, \underline{dw} \\
    &= \langle \mathcal{R}_\sigma f, h \rangle_{\mathbb{S}^{d-1} \times \mathbb{R}}.
\end{align*}
The duality identity is thus rigorously satisfied on the Schwartz space, completing the proof.
\end{proof}

\section{Distributional Inversion and Neural Representation via the Regularized Spectrum}\label{sec:caracterization_barron_distrib}

Now, we introduce the notion of the Generalized Inversion Filter.
\begin{definition}[Generalized Inversion Filter $\mathcal{T}_g$]
We define the generalized inversion filter $\mathcal{T}_g$ (operating on the scalar bias variable $b$) formally by its Fourier transform in the spectral domain:
\begin{equation}
    \widehat{\mathcal{T}_g}(\rho) = \text{Fp} \left( \frac{|\mathbb{S}^{d-1}|}{2(2\pi)^{d-1}} \frac{|\rho|^{d-1}}{|\widehat{g}(\rho)|^2} \right)
\end{equation}
in the sense of distributions. The notation $\text{Fp}$ designates the finite part in the sense of Hadamard.
\end{definition}

The following proposition gives sufficient conditions for this filter to define a tempered distribution.
\begin{proposition}[Sufficient conditions for the existence of $\mathcal{T}_g$ in $\mathcal{S}'(\mathbb{R})$]\label{prop:existence_Tg}
The generalized inversion filter $\mathcal{T}_g$ is a well-defined tempered distribution, i.e., $\mathcal{T}_g \in \mathcal{S}'(\mathbb{R})$, if the spectrum of the activation $\widehat{g}$ (assumed to be smooth) satisfies the following three conditions:
\begin{enumerate}
    \item \textbf{Frequency completeness:} $\widehat{g}(\rho) \neq 0$ for almost all $\rho \in \mathbb{R}$.
    \item \textbf{Moderate decay at high frequency:} There exist constants $R > 0$, $C > 0$, and $N > 0$ such that $\widehat{g}(\rho) \neq 0$ for all $|\rho| \ge R$, and:
    \begin{equation}\label{eq:g_polynomial_decay_infty}
        |\widehat{g}(\rho)| \ge C |\rho|^{-N} \quad \text{for all } |\rho| \ge R.
    \end{equation}
    \item \textbf{Tempered isolated singularities:} On the compact interval $[-R, R]$, the function $\widehat{g}$ possesses at most a finite number of isolated zeros, each of finite algebraic multiplicity.
\end{enumerate}
\end{proposition}

\begin{proof}
The operator $\mathcal{T}_g$ belongs to $\mathcal{S}'(\mathbb{R})$ as soon as its Fourier transform, defined by the quotient $S(\rho) = \frac{|\mathbb{S}^{d-1}|}{2(2\pi)^{d-1}} \frac{|\rho|^{d-1}}{|\widehat{g}(\rho)|^2}$, defines a valid tempered distribution.

By Condition 1, $S(\rho)$ is defined almost everywhere. Condition 2 guarantees that $S(\rho)$ exhibits at most polynomial growth at infinity, which is the standard sufficient condition to generate a continuous linear functional on $\mathcal{S}(\mathbb{R})$ away from local singularities \cite[Chap. VII]{schwartz1966theorie}, \cite[Chap. 3]{zemanian1987distribution}.

To regularize local non-integrabilities, Condition 3 ensures the singularities are isolated and of finite algebraic order. Hence, the Hadamard finite value ($\text{Fp}$) extension evaluates Taylor polynomials up to a strictly fixed maximum degree. This guarantees the resulting functional is controlled by a finite number of Schwartz semi-norms, fulfilling the essential topological prerequisite for $\mathcal{S}'(\mathbb{R})$ \cite[Chap. I]{gelfand1964generalized}. 

Since $\widehat{\mathcal{T}_g} \in \mathcal{S}'(\mathbb{R})$, the topological isomorphism of the Fourier transform immediately yields $\mathcal{T}_g \in \mathcal{S}'(\mathbb{R})$.
\end{proof}

We summarize our structural assumptions on the activation function $\sigma$ ensuring that $\mathcal{T}_g$ is a well-defined tempered distribution:
\begin{hypothesis}[Admissible Activation Functions]\label{hyp:complete_hyp_sigma}
The activation $\sigma \in \mathcal{S}'(\mathbb{R})$ is a $C^0(\mathbb{R})$ function with at most polynomial growth of order $k_\sigma$. Furthermore, we assume there exists an integer differentiation order $\alpha \in \mathbb{N}^*$ such that the derived distribution $g = D^\alpha \sigma$, whose spectrum is given by $\widehat{g}(\rho) = (i\rho)^\alpha \widehat{\sigma}(\rho)$, strictly satisfies the following properties:
\begin{enumerate}
    \item \textbf{Infinite Spectral Regularity:} The regularized spectrum is a smooth function, meaning $\widehat{g} \in C^\infty(\mathbb{R})$.
    \item \textbf{Filter Existence:} $\widehat{g}$ satisfies all three structural conditions required by Proposition \ref{prop:existence_Tg} (Frequency completeness, Moderate decay at high frequency, and Tempered singularities).
\end{enumerate}
\end{hypothesis}

\begin{remark}
It is crucial to note that Hypothesis \ref{hyp:complete_hyp_sigma} is verified for general activation functions used in Deep Learning.

\textbf{1. Modern Deep Learning Activations ($\alpha \in \mathbb{N}^*$):} \\
These standard neural network activations do not decay at infinity, which generates severe singularities at the frequency origin. The regularized autocorrelation strictly absorbs these singularities via the distributional derivative $g = D^\alpha \sigma$.
\begin{itemize}
    \item \textbf{The Rectified Linear Unit (ReLU):} $\sigma(t) = \max(0, t)$ \\
    Its raw spectrum contains heavy singularities. However, taking the second distributional derivative ($\alpha = 2$) yields the Dirac delta $g(t) = \delta(t)$. Its spectrum is the constant function $\widehat{g}(\rho) = 1$, yielding a perfect constant power density $|\widehat{g}(\rho)|^2 = 1 \in \mathcal{S}'(\mathbb{R})$.
    
    \item \textbf{The Heaviside Step Function (Binary Networks):} $\sigma(t) = \mathbf{1}_{t \ge 0}$ \\
    While possessing a jump discontinuity at the origin (relaxing the $C^0$ continuity to the bounded variation setting), a single derivative ($\alpha = 1$) isolates the jump, giving $g(t) = \delta(t)$. Consequently, the regularized spectrum is again $\widehat{g}(\rho) = 1$, ensuring validity.

    \item \textbf{The Leaky ReLU:} $\sigma(t) = \max(\gamma t, t)$ for $\gamma > 0$, $\gamma \neq 1$ \\
    Similar to the standard ReLU, it acts as a piecewise linear function. A second derivative ($\alpha = 2$) yields $g(t) = (1 - \gamma)\delta(t)$. The regularized spectrum is the constant $\widehat{g}(\rho) = 1 - \gamma$, making the squared modulus perfectly defined.

    \item \textbf{The Absolute Value function:} $\sigma(t) = |t|$ \\
    Similar to ReLU, taking the second distributional derivative ($\alpha = 2$) yields $g(t) = 2\delta(t)$. Its spectrum is the constant function $\widehat{g}(\rho) = 2$, which ensures a valid regularized tempered autocorrelation.

    \item \textbf{The Exponential Linear Unit (ELU):} $\sigma(t) = t$ if $t \ge 0$, and $\gamma(e^t - 1)$ if $t < 0$ \\
    Taking the second derivative ($\alpha = 2$) yields $g(t) = (1-\gamma)\delta(t) + \gamma \mathbf{1}_{t<0} e^t$. Its spectrum is $\widehat{g}(\rho) = (1-\gamma) + \frac{\gamma}{1-i\rho}$, which has a rational behavior at high frequencies and tends to a non-zero constant $1-\gamma$ (or decays as $1/\rho$ if $\gamma=1$). This ensures the moderate decay condition is perfectly satisfied.
\end{itemize}

\vspace{0.3cm}
\textbf{2. Activation functions that DO NOT verify the hypothesis:} \\
Conversely, some well-known activation functions fail to satisfy the required spectral conditions.
\begin{itemize}
    \item \textbf{Smooth saturating activations (Sigmoid, Tanh):} \\
    For example, the Sigmoid $\sigma(t) = 1/(1+e^{-t})$ is infinitely differentiable. As a result, its derivatives decay exponentially fast in the frequency domain (e.g., $\widehat{D\sigma}(\rho) = \frac{\pi \rho}{\sinh(\pi \rho)}$). This rapid exponential decay fails the moderate decay condition (Condition 2), meaning the generalized inversion filter would require exponential amplification, making it ill-posed in $\mathcal{S}'(\mathbb{R})$.
    
    \item \textbf{Smooth non-saturating activations (SiLU/Swish, GELU, Softplus):} \\
    Similar to the Sigmoid, these functions are smooth ($C^\infty$) everywhere. Thus, their spectra decay faster than any polynomial at infinity. They violate the moderate decay condition (exponential non-smoothness), preventing the existence of a tempered inversion filter.
    
    \item \textbf{Polynomial activations (e.g., $\sigma(t) = t^2$ or $t^3$):} \\
    Taking successive derivatives of a polynomial eventually yields a constant, and then zero. The spectrum of such functions is composed entirely of Dirac deltas and their derivatives, supported only at the frequency origin. This violates the frequency completeness condition (Condition 1), as the spectrum $\widehat{g}(\rho)$ is exactly zero almost everywhere.

    \item \textbf{The Triangle function (B-spline of order 1):} \\
    Consider $\sigma(t) = \max(0, 1 - |t|)$. Its Fourier transform is $\widehat{\sigma}(\rho) = \mathrm{sinc}^2(\rho / 2) = \frac{4 \sin^2(\rho / 2)}{\rho^2}$. It vanishes periodically at $\rho_k = 2k\pi$ for all $k \in \mathbb{Z} \setminus \{0\}$. Consequently, the lower bound $|\widehat{\sigma}(\rho)| \ge C |\rho|^{-N}$ fails as $|\rho| \to \infty$, and the quotient $S(\rho) = |\rho|^{d-1} / |\widehat{\sigma}(\rho)|^2$ exhibits an infinite sequence of poles of order 4 stretching to infinity. The Triangle function is therefore strictly non-admissible under Hypothesis \ref{hyp:complete_hyp_sigma}.
\end{itemize}
\end{remark}

\begin{definition}[Filtered Analysis Operator]\label{def:filter}
Let $f \in \mathcal{S}(\mathbb{R}^d)$.
The tempered distribution (parameterized by $w$) formally denoted by $(- \Delta_b)^{\alpha} \mathcal{T}_g *_b \mathcal{R}_\sigma f(w,\cdot)$ is formally motivated by the spectral product $(\rho)^{2\alpha}\widehat{\mathcal{T}_g}(\rho) \cdot \mathcal{F}_b\{\mathcal{R}_\sigma f\}(w, \rho)$, and is rigorously defined as the unique tempered distribution $H_w \in \mathcal{S}'(\mathbb{R})$ whose 1D spectrum is given by:
\begin{equation}\label{eq:def_hw}
    \widehat{H_w}(\rho) :=  \text{Fp} \left( (i\rho)^{\alpha} \frac{|\mathbb{S}^{d-1}|}{2(2\pi)^{d-1}} \frac{|\rho|^{d-1}}{\widehat{g}(\rho)} \widehat{f}(\rho w) \right),
\end{equation}
which is a well-defined element of $\mathcal{S}'(\mathbb{R})$ since $g$ satisfies the conditions of Proposition \ref{prop:existence_Tg}.
\end{definition}

\begin{remark}[From Filter Existence to Exact Inversion]
While Proposition \ref{prop:existence_Tg} guarantees that the inversion filter $\mathcal{T}_g$ and the analysis distribution $H_w$ can be defined via Hadamard finite parts even in the presence of isolated zeros, exact neural reconstruction requires inverting the frequency multiplier without loss of information. To prevent frequency nulls from annihilating target spectral components, we now formalize this non-vanishing requirement.
\end{remark}

\begin{hypothesis}[Non-Vanishing Regularized Spectrum]\label{hyp:non_zero_g}
In addition to Hypothesis \ref{hyp:complete_hyp_sigma}, to guarantee exact neural reconstruction, we assume that the regularized spectrum $\widehat{g}$ does not vanish on the real line:
\begin{equation}
    \forall \rho \in \mathbb{R}, \quad \widehat{g}(\rho) \neq 0.
\end{equation}
\end{hypothesis}

We now establish a key spectral convolution identity on the bias fibers, which will be used repeatedly in our representation results.
\begin{lemma}[Spectral Convolution Identity]\label{lem:spectral_convolution_identity}
Let $\sigma \in \mathcal{S}'(\mathbb{R})$ be a continuous activation function with at most polynomial growth of order $k_\sigma$. Let $\mu$ be a signed Radon measure on $\mathbb{R}$ with finite weighted total variation of order $p \geq k_\sigma + 2$, meaning:
\begin{equation}
    \int_{\mathbb{R}} (1 + |b|)^p \, d|\mu|(b) < \infty.
\end{equation}
Let $v(s) := (\mu *_b \sigma)(s) = \int_{\mathbb{R}} \sigma(s-b) \, d\mu(b)$ be the 1D spatial convolution. Then $v \in \mathcal{S}'(\mathbb{R})$ and its distributional Fourier transform satisfies:
\begin{equation}
    \widehat{v} = \widehat{\mu} \widehat{\sigma}
\end{equation}
in the sense of tempered distributions on $\mathbb{R}$.
\end{lemma}

\begin{proof}
Let $\varphi \in \mathcal{S}(\mathbb{R})$ be an arbitrary test function. By definition of the Fourier transform on tempered distributions, the action of the spectrum is evaluated as:
\begin{equation*}
    \langle \widehat{v}, \varphi \rangle_{\mathbb{R}} = \langle v, \widehat{\varphi} \rangle_{\mathbb{R}} = \int_{\mathbb{R}} \left( \int_{\mathbb{R}} \sigma(s-b) \, d\mu(b) \right) \widehat{\varphi}(s) \, ds.
\end{equation*}

We first justify the application of Fubini's theorem. Because the activation $\sigma$ has polynomial growth of order $k_\sigma$, we can bound the integrand: $|\sigma(s-b)| \le C(1+|s-b|)^{k_\sigma} \le C(1+|s|)^{k_\sigma}(1+|b|)^{k_\sigma}$. Consequently, the absolute double integral is bounded by:
\begin{equation*}
    \int_{\mathbb{R}} \int_{\mathbb{R}} |\sigma(s-b)| \, d|\mu|(b) \, |\widehat{\varphi}(s)| \, ds \leq C \left( \int_{\mathbb{R}} (1+|b|)^{k_\sigma} d|\mu|(b) \right) \left( \int_{\mathbb{R}} (1+|s|)^{k_\sigma} |\widehat{\varphi}(s)| \, ds \right).
\end{equation*}
Since $\widehat{\varphi} \in \mathcal{S}(\mathbb{R})$, its weighted integral is finite. Moreover, since $\mu$ has finite weighted total variation of order $p \geq k_\sigma + 2 > k_\sigma$, the first integral is strictly bounded. Fubini's theorem therefore strictly applies.

Swapping the integrals and performing the change of variable $x = s-b$ (with $ds = dx$), we obtain:
\begin{equation*}
    \langle \widehat{v}, \varphi \rangle_{\mathbb{R}} = \int_{\mathbb{R}} \sigma(x) \left( \int_{\mathbb{R}} \widehat{\varphi}(x+b) \, d\mu(b) \right) dx = \langle \sigma, \eta \rangle_{\mathbb{R}},
\end{equation*}
where we define the inner test function $\eta(x) := \int_{\mathbb{R}} \widehat{\varphi}(x+b) \, d\mu(b)$. 

Next, we expand the Fourier transform $\widehat{\varphi}(x+b) = \int_{\mathbb{R}} \varphi(\rho) e^{-i\rho(x+b)} \, d\rho$ and substitute it into $\eta(x)$. Applying Fubini's theorem once more to swap the integrals over $b$ and $\rho$:
\begin{equation*}
    \eta(x) = \int_{\mathbb{R}} \varphi(\rho) e^{-i\rho x} \left( \int_{\mathbb{R}} e^{-i\rho b} \, d\mu(b) \right) d\rho.
\end{equation*}
The inner integral over $b$ is precisely the classical continuous Fourier transform of the measure $\mu$, denoted $\widehat{\mu}(\rho)$. Thus:
\begin{equation*}
    \eta(x) = \int_{\mathbb{R}} \left( \varphi(\rho) \widehat{\mu}(\rho) \right) e^{-i\rho x} \, d\rho.
\end{equation*}
This expression demonstrates that $\eta$ is exactly the Fourier transform of the constructed function $\theta(\rho) := \widehat{\mu}(\rho) \varphi(\rho)$. Because $\mu$ has finite weighted total variation of order $p$ (with $p \in \mathbb{N}$), its Fourier transform $\widehat{\mu}$ is a $C^p$ function with bounded derivatives. 

Substituting $\eta = \widehat{\theta}$ back into the spatial integral yields:
\begin{equation*}
    \langle \widehat{v}, \varphi \rangle_{\mathbb{R}} = \int_{\mathbb{R}} \sigma(x) \widehat{\theta}(x) \, dx.
\end{equation*}
Since $\mu$ has finite weighted total variation of order $p \ge k_\sigma + 2$, the test function $\theta = \widehat{\mu} \varphi$ belongs to $\mathcal{S}^p(\mathbb{R})$. Applying Lemma \ref{lem:fourier_order_poly_growth} from Appendix \ref{sec:appendix_distributions} to the continuous activation $\sigma$ (which exhibits polynomial growth of order $k_\sigma$), the distributional Fourier transform $\widehat{\sigma}$ has distribution order at most $k_\sigma + 2$, and the Fourier duality identity extends unconditionally to $\mathcal{S}^p(\mathbb{R})$:
\begin{equation*}
    \int_{\mathbb{R}} \sigma(x) \widehat{\theta}(x) \, dx = \langle \widehat{\sigma}, \theta \rangle_{\mathbb{R}} = \langle \widehat{\sigma}, \widehat{\mu} \varphi \rangle_{\mathbb{R}}.
\end{equation*}
Finally, by the definition of the product between the distribution $\widehat{\sigma}$ of order at most $p$ and the $C^p$ multiplier $\widehat{\mu}$ with bounded derivatives, we have $\langle \widehat{\sigma}, \widehat{\mu} \varphi \rangle = \langle \widehat{\mu} \widehat{\sigma}, \varphi \rangle$. Therefore:
\begin{equation*}
    \langle \widehat{v}, \varphi \rangle_{\mathbb{R}} = \langle \widehat{\mu} \widehat{\sigma}, \varphi \rangle_{\mathbb{R}}.
\end{equation*}
Since this equality holds for any arbitrary test function $\varphi \in \mathcal{S}(\mathbb{R})$, we rigorously conclude by direct identification that $\widehat{v} = \widehat{\mu} \widehat{\sigma}$ in $\mathcal{S}'(\mathbb{R})$.
\end{proof}

The preceding lemmas provide a rigorous mechanism to manipulate the network's parameter space directly in the frequency domain, completely bypassing the spectral singularities of the activation function. We are now in a position to demonstrate that any sufficiently regular target function can be exactly reconstructed by applying the dual synthesis operator $\mathcal{R}_\sigma^*$ to our uniquely defined filtered analysis distribution. This culminates in the following neural representation theorem in the distributional sense.

\begin{theorem}[Neural Representation and Generalized Regularity]\label{th:main_th_distrib}
Assume $\alpha + d - 2 \ge k_\sigma + 2$ with $\alpha \in \mathbb{N}^*$. Let $f \in \mathcal{S}(\mathbb{R}^d)$ and assume Hypothesis \ref{hyp:complete_hyp_sigma} together with Hypothesis \ref{hyp:non_zero_g}. 

Then, the canonical filtered analysis density:
\begin{equation}
    h(w, b) := (- \Delta_b)^{\alpha} (\mathcal{T}_g *_b \mathcal{R}_\sigma f) (w, b)
\end{equation}
belongs to the weighted space $L^1(\mathbb{S}^{d-1} \times \mathbb{R}; (1 + |b|)^{k_\sigma + 2} db \, \underline{dw})$, and $f$ admits the exact continuous neural representation:
\begin{equation}
    f(x) = \mathcal{R}_\sigma^* h(x) = \int_{\mathbb{S}^{d-1}} \int_{\mathbb{R}} h(w, b) \sigma(\langle w, x \rangle - b) \, db \, \underline{dw}.
\end{equation}
\end{theorem}

\begin{remark}[Automatic Weighted Integrability for Schwartz Targets]\label{rem:regularity_h_hat}
The weighted integrability condition $h \in L^1(\mathbb{S}^{d-1} \times \mathbb{R}; (1 + |b|)^{k_\sigma + 2} db \, \underline{dw})$ is an unconditional property for any target function $f \in \mathcal{S}(\mathbb{R}^d)$ under the dimensional threshold $\alpha + d - 2 \ge k_\sigma + 2$.

Indeed, for every orientation $w \in \mathbb{S}^{d-1}$, the 1D fiber spectrum reads:
\begin{equation*}
    \widehat{h}(w, \rho) = (i\rho)^{\alpha} \frac{|\mathbb{S}^{d-1}|}{2(2\pi)^{d-1}} \frac{|\rho|^{d-1}}{\widehat{g}(\rho)} \widehat{f}(\rho w).
\end{equation*}
At high frequencies ($|\rho| \to \infty$), the rapid Schwartz decay of $\widehat{f} \in \mathcal{S}(\mathbb{R}^d)$ dominates the polynomial growth of $1/\widehat{g}(\rho)$ (ensured by Hypothesis \ref{hyp:complete_hyp_sigma}), guaranteeing fast decay of all derivatives uniformly in $w$. 

Near the frequency origin ($\rho = 0$), the denominator $\widehat{g}(\rho)$ is smooth and non-vanishing by Hypothesis \ref{hyp:non_zero_g}. The behavior of $\widehat{h}(w, \cdot)$ is therefore entirely governed by the geometric factor $u(\rho) := (i\rho)^\alpha |\rho|^{d-1}$:
\begin{itemize}
    \item If $d$ is odd, $|\rho|^{d-1} = \rho^{d-1}$ is an even polynomial; hence $\widehat{h}(w, \cdot) \in \mathcal{S}(\mathbb{R})$ is infinitely differentiable on all of $\mathbb{R}$, which implies that $h(w, \cdot) \in \mathcal{S}(\mathbb{R})$ exhibits faster-than-polynomial spatial decay.
    
    \item If $d$ is even, $d-1$ is odd, so $u(\rho) = i^\alpha \rho^M \operatorname{sgn}(\rho)$ with $M := \alpha + d - 1$. Setting $F(\rho) := \widehat{h}(w, \rho)$, the function $F$ is smooth on $\mathbb{R} \setminus \{0\}$, belongs to $C^{M-1}(\mathbb{R})$ with $F^{(k)}(0) = 0$ for all $0 \le k \le M-1$, and its $M$-th one-sided derivatives at the origin satisfy:
    \begin{equation*}
        [F^{(M)}]_0 := \lim_{\rho \to 0^+} F^{(M)}(\rho) - \lim_{\rho \to 0^-} F^{(M)}(\rho) = 2 i^\alpha M! \frac{|\mathbb{S}^{d-1}|}{2(2\pi)^{d-1}} \frac{\widehat{f}(0)}{\widehat{g}(0)} \neq 0.
    \end{equation*}
    To establish the spatial decay of $h(w, b) = \frac{1}{2\pi}\int_{-\infty}^{+\infty} F(\rho) e^{i b \rho} \, d\rho$ as $|b| \to \infty$, we split the integral onto $(-\infty, 0)$ and $(0, +\infty)$. Integrating by parts $M+1$ times on each half-line (where $F$ is $C^\infty$):
    \begin{align*}
        \int_0^{+\infty} F(\rho) e^{i b \rho} \, d\rho &= \sum_{k=0}^M \left[ (-1)^k \frac{F^{(k)}(\rho)}{(ib)^{k+1}} e^{i b \rho} \right]_{0^+}^{+\infty} + \frac{(-1)^{M+1}}{(ib)^{M+1}} \int_0^{+\infty} F^{(M+1)}(\rho) e^{i b \rho} \, d\rho, \\
        \int_{-\infty}^0 F(\rho) e^{i b \rho} \, d\rho &= \sum_{k=0}^M \left[ (-1)^k \frac{F^{(k)}(\rho)}{(ib)^{k+1}} e^{i b \rho} \right]_{-\infty}^{0^-} + \frac{(-1)^{M+1}}{(ib)^{M+1}} \int_{-\infty}^0 F^{(M+1)}(\rho) e^{i b \rho} \, d\rho.
    \end{align*}
    Because $F^{(k)}(0^+) = F^{(k)}(0^-) = 0$ for all $0 \le k \le M-1$, all lower-order boundary terms vanish identically at $\rho = 0$, while the terms at $\pm\infty$ vanish by Schwartz decay. Recombining the two half-lines, only the boundary evaluation at $k = M$ survives:
    \begin{equation*}
        \int_{-\infty}^{+\infty} F(\rho) e^{i b \rho} \, d\rho = \frac{(-1)^{M+1}}{(ib)^{M+1}} [F^{(M)}]_0 + \frac{(-1)^{M+1}}{(ib)^{M+1}} \int_{\mathbb{R} \setminus \{0\}} F^{(M+1)}(\rho) e^{i b \rho} \, d\rho.
    \end{equation*}
    Since $F^{(M+1)} \in L^1(\mathbb{R} \setminus \{0\})$, the remaining integral is uniformly bounded in $b$, which rigorously yields the sharp asymptotic decay:
    \begin{equation*}
        |h(w, b)| = \mathcal{O}\left( \frac{1}{|b|^{M+1}} \right) = \mathcal{O}\left( \frac{1}{|b|^{\alpha+d}} \right) \quad \text{as } |b| \to \infty.
    \end{equation*}
    The weighted fiber integrand therefore behaves asymptotically as:
    \begin{equation*}
        (1 + |b|)^{k_\sigma + 2} |h(w, b)| = \mathcal{O}\left( \frac{1}{|b|^{\alpha + d - k_\sigma - 2}} \right).
    \end{equation*}
    Since $\alpha + d - 2 \ge k_\sigma + 2$, the exponent satisfies $\alpha + d - k_\sigma - 2 \ge 2 > 1$, which guarantees absolute integrability over $\mathbb{R}$.
\end{itemize}
Because these estimates hold uniformly over the compact sphere $\mathbb{S}^{d-1}$, the canonical analysis distribution $h$ unconditionally defines a genuine density in $L^1(\mathbb{S}^{d-1} \times \mathbb{R}; (1 + |b|)^{k_\sigma + 2} db \, \underline{dw})$.
\end{remark}

\begin{proof}

To demonstrate the equality $f = \mathcal{R}_\sigma^* h$ in the weak sense on $\mathcal{S}'(\mathbb{R}^d)$, let $\phi \in \mathcal{S}(\mathbb{R}^d)$ be an arbitrary test function. By definition of the dual action, we evaluate the duality bracket:
\begin{equation*}
    \langle \mathcal{R}_\sigma^* h, \phi \rangle_{\mathbb{R}^d} = \langle \mathcal{R}_\sigma \phi, h \rangle_{\mathbb{S}^{d-1} \times \mathbb{R}}.
\end{equation*}

To rigorously justify writing this bracket as a Lebesgue integral, we must establish that the smooth function $\mathcal{R}_\sigma \phi(w, b)$ possesses at most polynomial growth of order $k_\sigma$. We have $\mathcal{R}_\sigma \phi(w, b) = \langle \sigma(t), \psi_w(t+b) \rangle_t$, where the projected test function is $\psi_w(s) = \int_{\langle w, x \rangle = s} \phi(x) \, d\mathcal{H}^{d-1}(x) \in \mathcal{S}(\mathbb{R})$. According to Hypothesis \ref{hyp:complete_hyp_sigma}, the continuous activation $\sigma$ strictly satisfies $|\sigma(t)| \le C(1+|t|)^{k_\sigma}$. Using the algebraic inequality $(1+|s-b|)^{k_\sigma} \le 2^{k_\sigma}(1+|s|)^{k_\sigma}(1+|b|)^{k_\sigma}$, we can bound the transform evaluated at $t=s-b$:
\begin{equation*}
    |\mathcal{R}_\sigma \phi(w, b)| \le \int_{\mathbb{R}} |\sigma(s-b)| |\psi_w(s)| \, ds \le 2^{k_\sigma} C (1+|b|)^{k_\sigma} \int_{\mathbb{R}} (1+|s|)^{k_\sigma} |\psi_w(s)| \, ds.
\end{equation*}
Because the spatial test function $\phi$ belongs to $\mathcal{S}(\mathbb{R}^d)$, its 1D projections $\psi_w$ are uniformly bounded within the Schwartz space topology across all orientations $w \in \mathbb{S}^{d-1}$. This gives that $|\mathcal{R}_\sigma \phi(w, b)| \le M(1+|b|)^{k_\sigma}$. Since the density $h$ belongs to the weighted space $L^1(\mathbb{S}^{d-1} \times \mathbb{R}; (1 + |b|)^{k_\sigma + 2} db \, \underline{dw})$, their point-wise product is absolutely integrable. Consequently, the duality bracket strictly equates to the standard Lebesgue integral:
\begin{equation*}
    \langle \mathcal{R}_\sigma^* h, \phi \rangle_{\mathbb{R}^d} = \int_{\mathbb{S}^{d-1}} \int_{\mathbb{R}} h(w, b) \mathcal{R}_\sigma \phi(w, b) \, db \, \underline{dw}.
\end{equation*}

Next, recall that $\mathcal{R}_\sigma \phi(w, b) = \langle \sigma(t), \psi_w(t+b) \rangle_t$, where $\psi_w(s) = \int_{\langle w, x \rangle = s} \phi(x) \, d\mathcal{H}^{d-1}(x) \in \mathcal{S}(\mathbb{R})$ so that the integral becomes:
\begin{align*}
    \langle \mathcal{R}_\sigma^* h, \phi \rangle_{\mathbb{R}^d} &= \int_{\mathbb{S}^{d-1}} \left( \int_{\mathbb{R}} h(w, b) \langle \sigma(t), \psi_w(t+b) \rangle_t \, db \right) \underline{dw} \\
    &=  \int_{\mathbb{S}^{d-1}} \left( \int_{\mathbb{R}} h(w, b) \int_{\mathbb{R}} \sigma(t) \psi_w(t+b) \, dt \, db \right) \underline{dw}
\end{align*}
where we replaced the duality bracket with an integral as $\sigma$ is a function of polynomial growth. Next, we show that this integral is absolutely convergent: 

\begin{align*}
    \int_{\mathbb{S}^{d-1}} \int_{\mathbb{R}} \int_{\mathbb{R}} \left| h(w, b) \sigma(t) \psi_w(t+b) \right| \, dt \, db \, \underline{dw} &\leq \int_{\mathbb{S}^{d-1}} \int_{\mathbb{R}} \int_{\mathbb{R}} |h(w, b)| C(1 + |t|^{k_\sigma}) |\psi_w(t+b)| \, dt \, db \, \underline{dw} \\
    &\leq C \int_{\mathbb{S}^{d-1}} \int_{\mathbb{R}} \int_{\mathbb{R}} |h(w, b)| (1 + |t+b - b|^{k_\sigma}) |\psi_w(t+b)| \, dt \, db \, \underline{dw} \\
    &\leq C' \int_{\mathbb{S}^{d-1}} \int_{\mathbb{R}} \int_{\mathbb{R}} |h(w, b)| (1 + |t+b|^{k_\sigma} + |b|^{k_\sigma}) |\psi_w(t+b)| \, dt \, db \, \underline{dw} \\
    &= C' \int_{\mathbb{S}^{d-1}} \int_{\mathbb{R}} |h(w, b)| \left( \int_{\mathbb{R}} (1 + |s|^{k_\sigma} + |b|^{k_\sigma}) |\psi_w(s)| \, ds \right) db \, \underline{dw} \\
    &= C' \int_{\mathbb{S}^{d-1}} \int_{\mathbb{R}} |h(w, b)| (1 + |b|^{k_\sigma}) \left( \int_{\mathbb{R}} |\psi_w(s)| \, ds \right) db \, \underline{dw} \\
    &\quad + C' \int_{\mathbb{S}^{d-1}} \int_{\mathbb{R}} |h(w, b)| \left( \int_{\mathbb{R}} |s|^{k_\sigma} |\psi_w(s)| \, ds \right) db \, \underline{dw} \\
    &\leq C'' \int_{\mathbb{S}^{d-1}} \int_{\mathbb{R}} |h(w, b)| (1 + |b|^{k_\sigma + 2}) \, db \, \underline{dw} \\
    &= C'' \|h\|_{L^1(\mathbb{S}^{d-1} \times \mathbb{R}; (1 + |b|)^{k_\sigma + 2} db \, \underline{dw})} < \infty.
\end{align*}

Since the triple integral is absolutely convergent, Fubini's theorem allows us to safely rearrange the order of integration. Let us apply the change of variable $s = t+b$ (so $t = s-b$ and $dt = ds$) and group the variables to isolate a 1D spatial convolution:
\begin{equation*}
    \langle \mathcal{R}_\sigma^* h, \phi \rangle_{\mathbb{R}^d} = \int_{\mathbb{S}^{d-1}} \left( \int_{\mathbb{R}} \psi_w(s) \left( \int_{\mathbb{R}} h(w, b) \sigma(s-b) \, db \right) ds \right) \underline{dw}.
\end{equation*}

Let us define the inner convolution $v_w(s) := (h(w, \cdot) *_b \sigma)(s)$.  Since $h(w, \cdot) \in L^1(\mathbb{R}; (1+|b|)^{k_\sigma + 2}db)$, we can apply Lemma \ref{lem:spectral_convolution_identity} (with $p = k_\sigma + 2$) and $v_w$ is a valid tempered distribution. Moreover: 
\begin{equation}\label{eq:v_w}
\langle \mathcal{R}_\sigma^* h, \phi \rangle_{\mathbb{R}^d} =  \int_{\mathbb{S}^{d-1}}  \langle v_w(s), \psi_w(s) \rangle_s \, \underline{dw} =  \frac{1}{2\pi} \int_{\mathbb{S}^{d-1}}  \langle \widehat{v_w}(\rho), \widehat{\psi_w}(-\rho) \rangle_\rho \, \underline{dw}.
\end{equation}

Now we compute  $\widehat{v_w}(\rho)$ using Lemma \ref{lem:spectral_convolution_identity} again. It holds that:
\begin{equation*}
    \widehat{v_w}(\rho) = \widehat{h}(w, \rho) \widehat{\sigma}(\rho).
\end{equation*}
By substituting the explicit definition of $\widehat{h}(w, \rho)$, we obtain:
\begin{equation*}
    \widehat{v_w}(\rho) = \left( (i\rho)^{\alpha} \frac{|\mathbb{S}^{d-1}|}{2(2\pi)^{d-1}} \frac{|\rho|^{d-1}}{\widehat{g}(\rho)} \widehat{f}(\rho w) \right) \widehat{\sigma}(\rho).
\end{equation*}

To rigorously simplify this distributional product across the origin $\rho = 0$, we verify the hypotheses of Corollary \ref{cor:spectral_simplification} one by one:
\begin{enumerate}
    \item \textbf{Distribution and regularized spectrum:} The distribution $T = \widehat{\sigma} \in \mathcal{S}'(\mathbb{R})$ has finite order $m \le k_\sigma + 2$. The polynomial $P(\rho) = (i\rho)^\alpha$ satisfies $P \cdot \widehat{\sigma} = \widehat{g} \in C^\infty(\mathbb{R})$ with polynomial growth by Hypothesis \ref{hyp:complete_hyp_sigma}.
    
    \item \textbf{Factor $\Theta$:} Defining $\Theta(\rho) := \frac{|\mathbb{S}^{d-1}|}{2(2\pi)^{d-1}} \frac{|\rho|^{d-1}}{\widehat{g}(\rho)} \widehat{f}(\rho w)$, we have $\Theta \in C^\infty(\mathbb{R} \setminus \{0\})$ and $\Theta \in C^0(\mathbb{R})$ (since $d \ge 1$), with rapid decay at infinity inherited from $\widehat{f} \in \mathcal{S}(\mathbb{R}^d)$ and the non-vanishing smooth denominator $\widehat{g}$.
    
    \item \textbf{Regularity of $\Psi$ and vanishing derivatives at $\rho = 0$:} The composite multiplier is $\Psi(\rho) = P(\rho)\Theta(\rho) = \frac{|\mathbb{S}^{d-1}|}{2(2\pi)^{d-1}} \frac{u(\rho)}{\widehat{g}(\rho)} \widehat{f}(\rho w)$, where $u(\rho) := (i\rho)^\alpha |\rho|^{d-1}$. 
    For $\rho > 0$, $u(\rho) = i^\alpha \rho^{\alpha+d-1}$, while for $\rho < 0$, $u(\rho) = i^\alpha (-1)^{d-1} \rho^{\alpha+d-1}$. For every integer $0 \le k \le \alpha + d - 2$, the $k$-th derivative satisfies:
    \begin{equation*}
        \lim_{\rho \to 0^+} u^{(k)}(\rho) = \lim_{\rho \to 0^-} u^{(k)}(\rho) = 0,
    \end{equation*}
    which establishes that $u \in C^{\alpha+d-2}(\mathbb{R})$ with $u^{(k)}(0) = 0$ for all $0 \le k \le \alpha + d - 2$. Since $\frac{\widehat{f}(\cdot w)}{\widehat{g}} \in C^\infty(\mathbb{R})$, Leibniz's differentiation rule implies that $\Psi \in C^{\alpha+d-2}(\mathbb{R})$ and $\partial_\rho^k \Psi(0) = 0$ for all $0 \le k \le \alpha + d - 2$.
    
    \item \textbf{Order matching:} By hypothesis, $\alpha + d - 2 \ge k_\sigma + 2 \ge m = \mathrm{order}(\widehat{\sigma})$. Thus, $\Psi \in C^m(\mathbb{R})$ and $\partial_\rho^k \Psi(0) = 0$ for all $0 \le k \le m$.
\end{enumerate}
All conditions of Corollary \ref{cor:spectral_simplification} are strictly satisfied. Applying the corollary yields $\Psi \cdot \widehat{\sigma} = \Theta \cdot (P \cdot \widehat{\sigma}) = \Theta \cdot \widehat{g}$ in $\mathcal{S}'(\mathbb{R})$, so that the singular regularized denominator $\widehat{g}$ cancels out pointwise:
\begin{equation*}
    \widehat{v_w}(\rho) = \Theta(\rho) \widehat{g}(\rho) = \frac{|\mathbb{S}^{d-1}|}{2(2\pi)^{d-1}} |\rho|^{d-1} \widehat{f}(\rho w).
\end{equation*}

By the projection-slice theorem, the 1D spectrum of the projected spatial test function evaluates to $\widehat{\psi_w}(-\rho) = \widehat{\phi}(-\rho w)$. Since the spatial test function $\phi$ is real-valued, its spectrum satisfies Hermitian symmetry: $\widehat{\phi}(-\rho w) = \overline{\widehat{\phi}(\rho w)}$.
The 1D duality bracket resolves into an absolutely convergent Lebesgue integral, completely bypassing the singularities of $\widehat{\sigma}$:
\[
\langle v_w, \psi_w \rangle_{\mathbb{R}} = \frac{|\mathbb{S}^{d-1}|}{2(2\pi)^{d}} \int_{\mathbb{R}} |\rho|^{d-1} \widehat{f}(\rho w) \overline{\widehat{\phi}(\rho w)} \, d\rho.
\]

Reinjecting this result into \eqref{eq:v_w} yields:
\[
\langle \mathcal{R}_\sigma^* h, \phi \rangle_{\mathbb{R}^d} = \frac{|\mathbb{S}^{d-1}|}{2(2\pi)^d} \int_{\mathbb{S}^{d-1}} \int_{\mathbb{R}} |\rho|^{d-1} \widehat{f}(\rho w) \overline{\widehat{\phi}(\rho w)} \, d\rho \, \underline{dw}.
\]
To conclude, we perform a change of variables to Cartesian coordinates by setting $\xi = \rho w \in \mathbb{R}^d$. The double covering of the space by the parameter cylinder imposes the differential geometric relationship $d\rho \, \underline{dw} = \frac{2}{|\mathbb{S}^{d-1}|} \frac{d\xi}{|\xi|^{d-1}}$.
The singularity term cleanly cancels out:
\[
\langle \mathcal{R}_\sigma^* h, \phi \rangle_{\mathbb{R}^d} = \frac{1}{(2\pi)^d} \int_{\mathbb{R}^d} \widehat{f}(\xi) \overline{\widehat{\phi}(\xi)} \, d\xi.
\]
By the standard Plancherel identity on $\mathbb{R}^d$, this spectral integral is exactly equal to the spatial bracket $\langle f, \phi \rangle_{\mathbb{R}^d}$. Since this equality holds unconditionally for any arbitrary test function $\phi \in \mathcal{S}(\mathbb{R}^d)$, we rigorously conclude that $\mathcal{R}_\sigma^* h = f$ in the sense of tempered distributions.

\end{proof}

\section{Banach Formulation and the Space of Weighted Radon Measures}
\label{sec:caracterization_barron_tv}
In Section \ref{sec:caracterization_barron_distrib}, we established the neural representation theory in the distributional setting where the parameter space admits a smooth, absolutely continuous density $d\mu(w,b) = h(w,b) \, db \, \underline{dw}$ with $h \in L^1(\mathbb{S}^{d-1} \times \mathbb{R}; (1+|b|)^{k_\sigma+2} db \, \underline{dw})$. While the Schwartz space $\mathcal{S}(\mathbb{R}^d)$ provided the necessary topological framework to rigorously construct the generalized Radon operators and handle singular activations, requiring the target function $f$ to possess rapid decay of all derivatives is overly restrictive for practical applications. 

To bridge our harmonic analysis with continuous neural networks, we extend our representation results to target functions in weighted Lebesgue spaces $L^1_p(\mathbb{R}^d)$, while generalizing the parameter representation to the Banach space of weighted Radon measures $\mathcal{M}_p(\mathbb{S}^{d-1} \times \mathbb{R})$. This accommodates the polynomial growth of the activation function $\sigma$ in the dual synthesis operator $\mathcal{R}_\sigma^* : \mathcal{M}_p \to L^\infty_{-k_\sigma}(\mathbb{R}^d)$, allowing for directional concentrations and general measure representations while preserving the underlying regularized spectral mechanics.

To achieve this extension while preserving a strict topological duality framework---and crucially, without losing the ability to invoke the Riesz-Markov-Kakutani representation theorem for the network's weights---we transition from the weak topology of distributions to weighted Banach spaces. We define a bias weight function $W_p(b) = (1 + |b|)^p$ with an order $p \geq k_\sigma$ and $\mathcal{M}_{p}(\mathbb{S}^{d-1} \times \mathbb{R})$ as the space of signed Radon measures with finite weighted total variation:

\begin{equation}
    \|\mu\|_{\mathrm{TV}, p} = \int_{\mathbb{S}^{d-1} \times \mathbb{R}} W_p(b) \, d|\mu|(w, b) < \infty.
\end{equation}
This space acts as the exact topological dual of $C_{0, p}(\mathbb{S}^{d-1} \times \mathbb{R})$, the space of continuous functions $u$ such that the quotient $u(w, b) / W_p(b)$ vanishes at infinity. The canonical norm on this space is given by $\|u\|_{C_{0,p}} = \sup_{(w,b)} \frac{|u(w,b)|}{W_p(b)}$.

\begin{lemma}[Weighted Riesz-Markov-Kakutani Theorem]
The space $\mathcal{M}_p(\mathbb{S}^{d-1} \times \mathbb{R})$ is isometrically isomorphic to the topological dual of $C_{0,p}(\mathbb{S}^{d-1} \times \mathbb{R})$.
\end{lemma}

\begin{proof}
Let $X = \mathbb{S}^{d-1} \times \mathbb{R}$. We construct an isometric isomorphism between the weighted space $C_{0,p}(X)$ and the classical space $C_0(X)$ of continuous functions vanishing at infinity. 

Define the mapping $T : C_{0,p}(X) \to C_0(X)$ by $T(u) = u / W_p$. By the very definition of the spaces and their respective norms, $T$ is an isometric isomorphism, meaning $\|T(u)\|_\infty = \|u\|_{C_{0,p}}$.

Let $L \in (C_{0,p}(X))^*$ be an arbitrary bounded linear functional. The composition $\tilde{L} = L \circ T^{-1}$ defines a bounded linear functional on $C_0(X)$. By the classical Riesz-Markov-Kakutani representation theorem, there exists a unique finite Radon measure $\nu \in \mathcal{M}(X)$ such that for all $v \in C_0(X)$:
\begin{equation*}
    \tilde{L}(v) = \int_X v(w,b) \, d\nu(w,b), \quad \text{with } \|\tilde{L}\|_{(C_0)^*} = \|\nu\|_{\mathrm{TV}}.
\end{equation*}

We can therefore express the action of $L$ on any test function $u \in C_{0,p}(X)$ by evaluating $\tilde{L}$ on $T(u)$:
\begin{equation*}
    L(u) = \tilde{L}(T(u)) = \int_X \frac{u(w,b)}{W_p(b)} \, d\nu(w,b).
\end{equation*}

Let us define the signed measure $\mu$ by its density with respect to $\nu$, specifically $d\mu = \frac{1}{W_p} d\nu$. The integral representation immediately becomes $L(u) = \int_X u(w,b) \, d\mu(w,b)$. 

It remains to verify that $\mu \in \mathcal{M}_p(X)$ and that the global isometry holds. We compute its weighted total variation:
\begin{equation*}
    \|\mu\|_{\mathrm{TV}, p} = \int_X W_p(b) \, d|\mu|(w,b) = \int_X W_p(b) \frac{1}{W_p(b)} \, d|\nu|(w,b) = \int_X d|\nu| = \|\nu\|_{\mathrm{TV}}.
\end{equation*}

Since $\nu$ is a finite measure, $\|\mu\|_{\mathrm{TV}, p} < \infty$, strictly confirming that $\mu \in \mathcal{M}_p(X)$. Furthermore, because $T$ is an isometry, $\|\tilde{L}\|_{(C_0)^*} = \|L\|_{(C_{0,p})^*}$, which immediately yields $\|L\|_{(C_{0,p})^*} = \|\mu\|_{\mathrm{TV}, p}$. This rigorously establishes the exact isometric duality $(C_{0,p}(X))^* \cong \mathcal{M}_p(X)$.
\end{proof}

\subsection{Topological Redefinition of the Analysis Operator $\mathcal{R}_\sigma$}

To ensure that the transform of a target function strictly belongs to the pre-dual $C_{0, p}(\mathbb{S}^{d-1} \times \mathbb{R})$, we must constrain the analysis domain using spatial weights.

\begin{definition}[Radon Transform towards Weighted $C_0$]
Let $\sigma \in C^0(\mathbb{R})$ be a continuous activation function with at most polynomial growth of order $k_\sigma$. Let $p > k_\sigma$ be the chosen weight parameter.
For a target function $f$ belonging to the spatially weighted Lebesgue space $L^1(\mathbb{R}^d; (1+|x|)^p dx)$, the analysis operator $\mathcal{R}_\sigma : L^1_{p}(\mathbb{R}^d) \to C_{0, p}(\mathbb{S}^{d-1} \times \mathbb{R})$ is defined by the absolute Lebesgue integral:
\begin{equation}
    \mathcal{R}_\sigma f(w, b) = \int_{\mathbb{R}^d} f(x) \sigma(\langle w, x \rangle - b) \, dx.
\end{equation}
\end{definition}

\begin{proposition}[Weighted Regularity and Asymptotic Vanishing]
Under these assumptions, the image $\mathcal{R}_\sigma f$ indeed belongs to $C_{0, p}(\mathbb{S}^{d-1} \times \mathbb{R})$.
\end{proposition}

\begin{proof}
Let $f \in L^1_p(\mathbb{R}^d)$ with $p > k_\sigma$. By the standard polynomial inequality $(1 + |\langle w, x \rangle - b|)^{k_\sigma} \le 2^{k_\sigma} (1 + |x|)^{k_\sigma} (1 + |b|)^{k_\sigma}$ and the polynomial growth bound on $\sigma$, the integrand is majorized for every $(w, b) \in \mathbb{S}^{d-1} \times \mathbb{R}$ by:
\begin{equation*}
    |f(x)| |\sigma(\langle w, x \rangle - b)| \le C 2^{k_\sigma} |f(x)| (1+|x|)^{k_\sigma} (1+|b|)^{k_\sigma}.
\end{equation*}
Since $f \in L^1_p(\mathbb{R}^d)$ with $p > k_\sigma$, $x \mapsto |f(x)|(1+|x|)^{k_\sigma}$ belongs to $L^1(\mathbb{R}^d)$. As $\sigma$ is continuous and the integrand is locally bounded by an integrable function uniformly on compact subsets of $\mathbb{S}^{d-1} \times \mathbb{R}$, the continuity theorem for parametric integrals ensures that $\mathcal{R}_\sigma f$ is continuous on $\mathbb{S}^{d-1} \times \mathbb{R}$.

Next, we establish the uniform asymptotic decay. Integrating the pointwise bound over $\mathbb{R}^d$ yields, for all $w \in \mathbb{S}^{d-1}$ and $b \in \mathbb{R}$:
\begin{equation*}
    |\mathcal{R}_\sigma f(w, b)| \le C 2^{k_\sigma} (1+|b|)^{k_\sigma} \|f\|_{L^1_{k_\sigma}} \le C 2^{k_\sigma} (1+|b|)^{k_\sigma} \|f\|_{L^1_p}.
\end{equation*}
Dividing both sides by the weight $W_p(b) = (1+|b|)^p$, we obtain:
\begin{equation*}
    \sup_{w \in \mathbb{S}^{d-1}} \frac{|\mathcal{R}_\sigma f(w, b)|}{W_p(b)} \le C 2^{k_\sigma} \|f\|_{L^1_p} (1+|b|)^{-(p-k_\sigma)}.
\end{equation*}
Since $p > k_\sigma$, the right-hand side converges to $0$ as $|b| \to \infty$ independently of $w \in \mathbb{S}^{d-1}$. This guarantees that $\mathcal{R}_\sigma f(w, b) / W_p(b)$ vanishes at infinity uniformly on $\mathbb{S}^{d-1} \times \mathbb{R}$, and therefore $\mathcal{R}_\sigma f \in C_{0,p}(\mathbb{S}^{d-1} \times \mathbb{R})$.
\end{proof}

\subsection{Rigorous Construction of the Adjoint by Weak-$*$ Duality}

The synthesis operator natively integrates directly with respect to a weighted signed measure $d\mu$, absorbing the polynomial divergence of the activation.

\begin{proposition}[Dual Synthesis Operator $\mathcal{R}_\sigma^*$]

The adjoint operator $\mathcal{R}_\sigma^* : \mathcal{M}_{p}(\mathbb{S}^{d-1} \times \mathbb{R}) \to L^\infty_{-k_\sigma}(\mathbb{R}^d)$ (where the output space has a polynomial growth of order $k_\sigma$) is uniquely defined by the bilinear duality relationship:
\begin{equation}
    \langle \mathcal{R}_\sigma f, \mu \rangle_{C_{0, p}, \mathcal{M}_{p}} = \langle f, \mathcal{R}_\sigma^* \mu \rangle_{L^1_{p}, L^\infty_{-p}}.
\end{equation}
Its point-wise action is given by:
\begin{equation}
    \mathcal{R}_\sigma^* \mu(x) = \int_{\mathbb{S}^{d-1}} \int_{\mathbb{R}} \sigma(\langle w, x \rangle - b) \, d\mu(w, b).
\end{equation}
\end{proposition}

\begin{proof}
Let $f \in L^1_{p}(\mathbb{R}^d)$. The duality action on the cylinder expands to:
\begin{equation*}
    \langle \mathcal{R}_\sigma f, \mu \rangle = \int_{\mathbb{S}^{d-1} \times \mathbb{R}} \left( \int_{\mathbb{R}^d} f(x) \sigma(\langle w, x \rangle - b) \, dx \right) d\mu(w, b).
\end{equation*}
To justify Fubini-Tonelli, we evaluate the absolute integral using the structural upper bound of $\sigma$:
\begin{align*}
    \int \int |f(x)| |\sigma(\langle w, x \rangle - b)| \, dx \, d|\mu|(w,b) &\le C' \int \int |f(x)| (1+|x|)^{k_\sigma} (1+|b|)^{k_\sigma} \, dx \, d|\mu| \\
    &\le C' \left( \int_{\mathbb{R}^d} |f(x)| (1+|x|)^p \, dx \right) \left( \int (1+|b|)^p \, d|\mu| \right) \\
    &= C' \|f\|_{L^1_p} \|\mu\|_{\mathrm{TV}, p} < \infty.
\end{align*}
Since absolute integrability is strictly verified, we commute the integrals:
\begin{equation*}
    \langle \mathcal{R}_\sigma f, \mu \rangle = \int_{\mathbb{R}^d} f(x) \left[ \int_{\mathbb{S}^{d-1} \times \mathbb{R}} \sigma(\langle w, x \rangle - b) \, d\mu(w, b) \right] dx.
\end{equation*}
By direct identification, we extract the spatial function $\mathcal{R}_\sigma^* \mu(x)$. Furthermore, we naturally obtain the bound $|\mathcal{R}_\sigma^* \mu(x)| \le C'(1+|x|)^{k_\sigma} \|\mu\|_{\mathrm{TV}, p}$, proving that the network output behaves exactly with growth $k_\sigma$.
\end{proof}

\subsection{Measure Disintegration and Topological Representation in $\mathcal{M}_p$}
Throughout this section, we continue to work under Hypothesis \ref{hyp:non_zero_g} ($\widehat{g}(\rho) \neq 0$ for all $\rho \in \mathbb{R}$). Under this hypothesis, the Fourier transform of $H_w$ given by \eqref{eq:def_hw} is defined for all $f \in L^1_p(\mathbb{R}^d)$ (which implies that $\widehat{f}$ is a bounded $C^p$ function) in the sense of tempered distributions. 

\begin{definition}[Extension of the Filtered Analysis Operator to $L^1_p$]\label{def:def_Hw_Mp}
We say that the global composite distribution $(- \Delta_b)^{\alpha} \mathcal{T}_g *_b \mathcal{R}_\sigma f$ belongs to the space of weighted Radon measures $\mathcal{M}_{p}(\mathbb{S}^{d-1} \times \mathbb{R})$ if there exists a measure $\mu \in \mathcal{M}_{p}(\mathbb{S}^{d-1} \times \mathbb{R})$ such that for all joint test functions $\phi \in \mathcal{S}(\mathbb{S}^{d-1} \times \mathbb{R})$:
\[
\int_{\mathbb{S}^{d-1}} \langle H_w, \phi(w, \cdot) \rangle_{\mathbb{R}} \, \underline{dw} = \int_{\mathbb{S}^{d-1} \times \mathbb{R}} \phi(w, b) \, d\mu(w, b),
\]
where $H_w := (- \Delta_b)^{\alpha} \mathcal{T}_g *_b \mathcal{R}_\sigma f(w,\cdot) \in \mathcal{S}'(\mathbb{R})$ is the fiber distribution defined spectrally in \eqref{eq:def_hw}.
\end{definition}

In order to prove the main theorem of this section, we need to disintegrate $(-\Delta_b)^\alpha \mathcal{T}_g *_b \mathcal{R}_\sigma f \in \mathcal{M}_{p}(\mathbb{S}^{d-1} \times \mathbb{R})$ with respect to the Lebesgue measure on the sphere $\mathbb{S}^{d-1}$.

\begin{lemma}[Fiber Disintegration of the Measure]\label{lem:disintegration_mu}
Let $f \in L^1_p(\mathbb{R}^d)$ and assume that the composite filtered analysis distribution, denoted globally as $\mu := (-\Delta_b)^\alpha \mathcal{T}_g *_b \mathcal{R}_\sigma f$, satisfies the condition $\mu \in \mathcal{M}_{p}(\mathbb{S}^{d-1} \times \mathbb{R})$. 

Then, the global Radon measure $\mu$ admits a strict geometric disintegration with respect to the uniform spherical Lebesgue measure $\underline{dw}$. Specifically, there exists a $\underline{dw}$-almost everywhere uniquely defined family of 1D signed Radon measures $\{\mu_w\}_{w \in \mathbb{S}^{d-1}}$ on $\mathbb{R}$ such that:
\begin{equation}
    d\mu(w,b) = d\mu_w(b) \, \underline{dw}.
\end{equation}
Furthermore, for $\underline{dw}$-almost every orientation $w \in \mathbb{S}^{d-1}$, the fiber measure $\mu_w$ has finite weighted total variation, meaning $\mu_w \in \mathcal{M}_p(\mathbb{R})$.
\end{lemma}

\begin{proof}
We divide the proof into three distinct analytical steps.
\paragraph{Step 1: Absolute continuity of the marginal variation measure with respect to $\underline{dw}$.}
Let $\pi : \mathbb{S}^{d-1} \times \mathbb{R} \to \mathbb{S}^{d-1}$ denote the canonical projection $\pi(w, b) = w$, and let $\lambda := \pi_*|\mu|$ be the push-forward measure representing the marginal of the total variation on $\mathbb{S}^{d-1}$. We must establish that $\lambda$ is absolutely continuous with respect to the uniform spherical Lebesgue measure, i.e., $\lambda \ll \underline{dw}$.

Let $K \subset \mathbb{S}^{d-1}$ be an arbitrary compact subset such that $\underline{dw}(K) = 0$. By the outer regularity of the Lebesgue measure on $\mathbb{S}^{d-1}$, there exists a decreasing sequence of open neighborhoods $U_n \supset K$ satisfying $\lim_{n \to \infty} \underline{dw}(U_n) = 0$. Using smooth mollification, construct a sequence of cutoff functions $\alpha_n \in C^\infty(\mathbb{S}^{d-1})$ such that $\mathbf{1}_K \le \alpha_n \le \mathbf{1}_{U_n}$ and $0 \le \alpha_n(w) \le 1$ for all $w \in \mathbb{S}^{d-1}$.

Let $\Psi \in \mathcal{S}(\mathbb{S}^{d-1} \times \mathbb{R})$ be an arbitrary joint test function. Testing the global measure $\mu$ against the sequence of truncated test functions $\Phi_n(w,b) := \alpha_n(w)\Psi(w,b)$ yields, by Definition~\ref{def:def_Hw_Mp}:
\begin{equation}\label{eq:action_phin_step1}
    \int_{\mathbb{S}^{d-1} \times \mathbb{R}} \alpha_n(w) \Psi(w,b) \, d\mu(w,b) = \int_{\mathbb{S}^{d-1}} \alpha_n(w) \langle H_w, \Psi(w, \cdot) \rangle_{\mathbb{R}} \, \underline{dw}.
\end{equation}

To control the right-hand side, we study the scalar angular mapping $F_\Psi : \mathbb{S}^{d-1} \to \mathbb{C}$ defined by $F_\Psi(w) := \langle H_w, \Psi(w, \cdot) \rangle_{\mathbb{R}}$.We reformulate this duality bracket in the 1D frequency domain:
\begin{equation*}
    F_\Psi(w) = \frac{1}{2\pi} \left\langle \widehat{H_w}, \, \mathcal{F}_b\{\Psi(w, \cdot)\}(-\cdot) \right\rangle_{\mathbb{R}},
\end{equation*}
where $\mathcal{F}_b\{\Psi(w, \cdot)\}(\rho) = \int_{\mathbb{R}} \Psi(w,b) e^{-i\rho b} \, db \in \mathcal{S}(\mathbb{S}^{d-1} \times \mathbb{R})$.
Recalling the spectral definition of $H_w$ from \eqref{eq:def_hw}, we isolate the direction-independent distribution $T \in \mathcal{S}'(\mathbb{R})$:
\begin{equation*}
    T(\rho) :=  (i\rho)^\alpha \frac{|\mathbb{S}^{d-1}|}{2(2\pi)^{d-1}} \frac{|\rho|^{d-1}}{\widehat{g}(\rho)} .
\end{equation*}
Thus, $F_\Psi(w) = \frac{1}{2\pi} \langle T, \theta_w \rangle_{\mathbb{R}}$, where $\theta_w(\rho) := \widehat{f}(\rho w) \cdot \mathcal{F}_b\{\Psi(w, \cdot)\}(-\rho)$. 
Here, by Hypothesis~\ref{hyp:non_zero_g}, the regularized spectrum $\widehat{g}(\rho)$ never vanishes on $\mathbb{R}$. By the moderate high-frequency decay condition on $\widehat{g}$, $T(\rho)$ exhibits at most polynomial growth at infinity. Thus, $T$ is a regular distribution of order $0$, and the duality bracket $\langle T, \theta_w \rangle_{\mathbb{R}}$ is well-defined as a standard convergent Lebesgue integral $\int_{\mathbb{R}} T(\rho) \theta_w(\rho) \, d\rho$.

 The continuity of $w \mapsto F_\Psi(w)$ follows directly from the standard theorem on the continuity of parametric integrals:
\begin{enumerate}
    \item \textbf{Pointwise continuity:} For almost every $\rho \in \mathbb{R}$, the mapping $w \mapsto T(\rho) \theta_w(\rho)$ is continuous on $\mathbb{S}^{d-1}$ since $\widehat{f} \in C^0(\mathbb{R}^d)$ and $\Psi \in \mathcal{S}(\mathbb{S}^{d-1} \times \mathbb{R})$.
    \item \textbf{Integrable domination:} Since $T(\rho)$ has at most polynomial growth of order $k \in \mathbb{N}$ (i.e., $|T(\rho)| \le C_1 (1+|\rho|)^k$) and $\|\widehat{f}\|_{L^\infty} \le \|f\|_{L^1} < \infty$, we bound the integrand pointwise for all $w \in \mathbb{S}^{d-1}$:
    \begin{equation*}
        |T(\rho) \theta_w(\rho)| \le C_1 \|\widehat{f}\|_{L^\infty} (1+|\rho|)^k \sup_{w' \in \mathbb{S}^{d-1}} \left| \mathcal{F}_b\{\Psi(w', \cdot)\}(-\rho) \right|.
    \end{equation*}
    Because $\Psi \in \mathcal{S}(\mathbb{S}^{d-1} \times \mathbb{R})$, the envelope function $g(\rho) := \sup_{w' \in \mathbb{S}^{d-1}} |\mathcal{F}_b\{\Psi(w', \cdot)\}(-\rho)|$ belongs to $\mathcal{S}(\mathbb{R})$ and decays rapidly as $|\rho| \to \infty$. Thus, the bounding function $\rho \mapsto C_1 \|\widehat{f}\|_{L^\infty} (1+|\rho|)^k g(\rho)$ is independent of $w$ and strictly integrable over $\mathbb{R}$.
\end{enumerate}
By the Dominated Convergence Theorem for parametric integrals, $F_\Psi$ is continuous on $\mathbb{S}^{d-1}$. Hence, since $F_\Psi \in C^0(\mathbb{S}^{d-1})$ on the compact manifold $\mathbb{S}^{d-1}$, it is globally bounded by $M_\Psi := \sup_{w \in \mathbb{S}^{d-1}} |F_\Psi(w)| < \infty$.

We can now apply the Cauchy--Schwarz inequality in $L^2(\mathbb{S}^{d-1}, \underline{dw})$ to the integral in \eqref{eq:action_phin_step1}:
\begin{align*}
    \left| \int_{\mathbb{S}^{d-1}} \alpha_n(w) F_\Psi(w) \, \underline{dw} \right| &\le \left( \int_{\mathbb{S}^{d-1}} |\alpha_n(w)|^2 \, \underline{dw} \right)^{1/2} \left( \int_{\mathbb{S}^{d-1}} |F_\Psi(w)|^2 \, \underline{dw} \right)^{1/2} \\
    &\le \left( \int_{\mathbb{S}^{d-1}} \mathbf{1}_{U_n}(w) \, \underline{dw} \right)^{1/2} M_\Psi = M_\Psi \sqrt{\underline{dw}(U_n)}.
\end{align*}
Taking the limit $n \to \infty$ gives $\lim_{n \to \infty} \int_{\mathbb{S}^{d-1} \times \mathbb{R}} \alpha_n(w)\Psi(w,b) \, d\mu(w,b) = 0$.

On the other hand, since $\alpha_n(w)\Psi(w,b) \to \mathbf{1}_K(w)\Psi(w,b)$ pointwise and is dominated by the $|\mu|$-integrable function $|\Psi(w,b)|$, Lebesgue's Dominated Convergence Theorem implies:
\begin{equation*}
    \int_{K \times \mathbb{R}} \Psi(w,b) \, d\mu(w,b) = 0 \quad \forall \Psi \in \mathcal{S}(\mathbb{S}^{d-1} \times \mathbb{R}).
\end{equation*}
By the density of $\mathcal{S}$ in $C_{0,p}$, the restriction of $\mu$ to the cylinder $K \times \mathbb{R}$ is identically zero. By the duality definition of total variation, its norm vanishes: $|\mu|(K \times \mathbb{R}) = 0$. Therefore, $\lambda(K) = \pi_*|\mu|(K) = |\mu|(K \times \mathbb{R}) = 0$. This rigorously proves that $\lambda \ll \underline{dw}$.

\paragraph{Step 2: Polar decomposition and disintegration of the total variation.}
By the polar decomposition theorem, there exists a Borel phase function $\theta : \mathbb{S}^{d-1} \times \mathbb{R} \to \mathbb{C}$ with $|\theta| = 1$ $|\mu|$-a.e. such that $d\mu = \theta \, d|\mu|$.

Since $|\mu|$ is a finite non-negative Radon measure whose marginal satisfies $\pi_*|\mu| \ll \underline{dw}$ (by Step~1), applying the disintegration theorem to $|\mu|$ with respect to $\underline{dw}$ yields a $\underline{dw}$-almost everywhere uniquely defined family of non-negative 1D Radon measures $\{|\mu_w|\}_{w \in \mathbb{S}^{d-1}}$ on $\mathbb{R}$ such that:
\begin{equation*}
    d|\mu|(w, b) = d|\mu_w|(b) \, \underline{dw}.
\end{equation*}
Defining the signed fiber measures by $d\mu_w(b) := \theta(w, b) \, d|\mu_w|(b)$, the condition $|\theta| = 1$ guarantees that the total variation of $\mu_w$ is precisely $|\mu_w|$, directly yielding:
\begin{equation*}
    d\mu(w, b) = d\mu_w(b) \, \underline{dw} \quad \text{and} \quad d|\mu|(w, b) = d|\mu_w|(b) \, \underline{dw}.
\end{equation*}

\paragraph{Step 3: Boundedness of the weighted total variation on the fibers.}
Using the identity $d|\mu|(w, b) = d|\mu_w|(b) \, \underline{dw}$, we evaluate the weighted total variation of $\mu$:
\begin{equation*}
    \|\mu\|_{\mathrm{TV}, p} = \int_{\mathbb{S}^{d-1} \times \mathbb{R}} (1+|b|)^p \, d|\mu|(w,b) = \int_{\mathbb{S}^{d-1}} \left( \int_{\mathbb{R}} (1+|b|)^p \, d|\mu_w|(b) \right) \underline{dw} < \infty.
\end{equation*}
Since the spherical integral of the non-negative measurable function $w \mapsto \int_{\mathbb{R}} (1+|b|)^p \, d|\mu_w|(b)$ is finite, this integrand must be finite for $\underline{dw}$-almost every $w \in \mathbb{S}^{d-1}$. Thus, there exists a Borel set of full measure $\Omega \subseteq \mathbb{S}^{d-1}$ such that for all $w \in \Omega$:
\begin{equation*}
    \|\mu_w\|_{\mathrm{TV}, p} = \int_{\mathbb{R}} (1+|b|)^p \, d|\mu_w|(b) < \infty.
\end{equation*}
This establishes that $\mu_w \in \mathcal{M}_p(\mathbb{R})$ for $\underline{dw}$-almost every orientation $w \in \mathbb{S}^{d-1}$, completing the proof.
\end{proof}

In fact, the disintegrated measure $\mu_w$ coincides with $H_w$ almost everywhere:

\begin{lemma}[Almost Everywhere Coincidence of Fiber Measures and Distributions]\label{lem:identification_spectrale_mu}
Let $\mu \in \mathcal{M}_p(\mathbb{S}^{d-1} \times \mathbb{R})$ be the global  measure associated with the composite analysis operator, and let $\{\mu_w\}_{w \in \mathbb{S}^{d-1}}$ be its disintegrated fiber measures satisfying $d\mu(w,b) = d\mu_w(b) \, \underline{dw}$ as established in Lemma \ref{lem:disintegration_mu}. Let $H_w \in \mathcal{S}'(\mathbb{R})$ be the directionally parameterized tempered distribution defined in Definition \ref{def:def_Hw_Mp}.

Then, for $\underline{dw}$-almost every orientation $w \in \mathbb{S}^{d-1}$,
\begin{equation*}
    \mu_w = H_w \quad \text{in } \mathcal{S}'(\mathbb{R}).
\end{equation*}
\end{lemma}

\begin{proof}
According to Lemma \ref{lem:disintegration_mu}, there exists a full-measure set $\Omega_0 \subset \mathbb{S}^{d-1}$ (i.e., $\int_{\mathbb{S}^{d-1} \setminus \Omega_0} \underline{dw} = 0$) such that for all $w \in \Omega_0$, the fiber measure $\mu_w$ belongs to $\mathcal{M}_p(\mathbb{R})$. Hence $\mu_w$ is a regular tempered distribution. Consequently, its 1D Fourier transform $\widehat{\mu}_w \in \mathcal{S}'(\mathbb{R})$ is unambiguously defined by standard duality for all $w \in \Omega_0$.

We must now prove that the distribution $\mu_w$ identifies strictly with $H_w$. By the global definition of the functional representing the composite operator, its action on any joint test function $\phi \in \mathcal{S}(\mathbb{S}^{d-1} \times \mathbb{R})$ evaluates to:
\begin{equation*}
    \langle \mu, \phi \rangle = \int_{\mathbb{S}^{d-1}} \langle H_w, \phi(w, \cdot) \rangle_{\mathbb{R}} \, \underline{dw}.
\end{equation*}

Simultaneously, evaluating this exact same action using the geometric disintegration of the measure yields:
\begin{equation*}
    \langle \mu, \phi \rangle = \int_{\mathbb{S}^{d-1} \times \mathbb{R}} \phi(w,b) \, d\mu(w,b) = \int_{\mathbb{S}^{d-1}} \langle \mu_w, \phi(w, \cdot) \rangle_{\mathbb{R}} \, \underline{dw}.
\end{equation*}

Equating these two representations, we obtain that for any joint test function $\phi$:
\begin{equation}\label{eq:integral_diff_zero}
    \int_{\mathbb{S}^{d-1}} \left( \langle H_w, \phi(w, \cdot) \rangle_{\mathbb{R}} - \langle \mu_w, \phi(w, \cdot) \rangle_{\mathbb{R}} \right) \underline{dw} = 0.
\end{equation}

To isolate the action on the fibers, we choose a separable test function of the form $\phi(w,b) = A(w)\psi(b)$, where $A \in C^\infty(\mathbb{S}^{d-1})$ is an arbitrary smooth angular function and $\psi \in \mathcal{S}(\mathbb{R})$ is a fixed 1D spatial test function. Substituting this into \eqref{eq:integral_diff_zero} gives:
\begin{equation*}
    \int_{\mathbb{S}^{d-1}} A(w) \Big( \langle H_w, \psi \rangle_{\mathbb{R}} - \langle \mu_w, \psi \rangle_{\mathbb{R}} \Big) \underline{dw} = 0.
\end{equation*}

Since the mapping $w \mapsto \langle H_w, \psi \rangle_{\mathbb{R}} - \langle \mu_w, \psi \rangle_{\mathbb{R}}$ belongs to $L^1(\mathbb{S}^{d-1}, \underline{dw})$ and its integral against every smooth test function $A \in C^\infty(\mathbb{S}^{d-1})$ vanishes, the standard $L^1 - C^\infty$ duality (or the density of smooth functions in $L^\infty(\mathbb{S}^{d-1})$) implies that this function vanishes $\underline{dw}$-almost everywhere. Thus, for this specific test function $\psi$, there exists a full-measure set $\Omega_\psi \subset \mathbb{S}^{d-1}$ such that for all $w \in \Omega_\psi$:
\begin{equation*}
    \langle H_w, \psi \rangle_{\mathbb{R}} = \langle \mu_w, \psi \rangle_{\mathbb{R}}.
\end{equation*}

However, the null set $\mathbb{S}^{d-1} \setminus \Omega_\psi$ depends inherently on the choice of $\psi$. To claim that the distributions are identical, the equality must hold for all $\psi \in \mathcal{S}(\mathbb{R})$ simultaneously on a single full-measure set. 
To rigorously resolve this, we exploit the fact that the Fréchet space $\mathcal{S}(\mathbb{R})$ is separable. Let $\{\psi_k\}_{k \in \mathbb{N}}$ be a countable dense subset of $\mathcal{S}(\mathbb{R})$. For each $k$, there exists a full-measure set $\Omega_k$ where $\langle H_w, \psi_k \rangle_{\mathbb{R}} = \langle \mu_w, \psi_k \rangle_{\mathbb{R}}$.

We construct the countable intersection $\Omega^* = \Omega_0 \cap \left( \bigcap_{k \in \mathbb{N}} \Omega_k \right)$. Since a countable union of measure-zero sets has measure zero, the intersection $\Omega^*$ remains a set of full measure ($\int_{\mathbb{S}^{d-1} \setminus \Omega^*} \underline{dw} = 0$).

For any fixed orientation $w \in \Omega^*$, the linear functionals $H_w$ and $\mu_w$ coincide on the dense subset $\{\psi_k\}$. Because both $H_w$ and $\mu_w$ are continuous on $\mathcal{S}(\mathbb{R})$, their equality uniquely extends to the entire space $\mathcal{S}(\mathbb{R})$. Therefore, we strictly have:
\begin{equation*}
    \mu_w = H_w \quad \text{in } \mathcal{S}'(\mathbb{R}) \text{ for } \underline{dw}\text{-almost every } w.
\end{equation*}
The proof is complete.
\end{proof}

Now we are able to state the main theorem of this section.

\begin{theorem}[Topological Representation Theorem in Weighted Radon Spaces]\label{th:carac_gen_barron_mp}
Assume $\alpha + d - 2 \ge k_\sigma + 2$. Let $f : \mathbb{R}^d \to \mathbb{R}$ belong to $L^1_p(\mathbb{R}^d)$ with $p \in \mathbb{N}$, $p \ge k_\sigma + 2$, and assume Hypothesis \ref{hyp:non_zero_g}. 

If the composite analysis distribution $(-\Delta_b)^\alpha \mathcal{T}_g *_b \mathcal{R}_\sigma f$ belongs to $\mathcal{M}_{p}(\mathbb{S}^{d-1} \times \mathbb{R})$ in the sense of Definition \ref{def:def_Hw_Mp}, then $f$ admits the continuous neural representation $f(x) = \mathcal{R}_\sigma^* \mu(x)$ generated by the finite weighted measure $\mu \in \mathcal{M}_{p}(\mathbb{S}^{d-1} \times \mathbb{R})$. 
\end{theorem}

\begin{proof}
Suppose that the composite distribution, which we denote by $\mu := (-\Delta_b)^\alpha \mathcal{T}_g *_b \mathcal{R}_\sigma f$, belongs to the space of weighted measures $\mathcal{M}_{p}(\mathbb{S}^{d-1} \times \mathbb{R})$.
Our goal is to demonstrate that $\mathcal{R}_\sigma^* \mu = f$ in the weak sense of distributions, which will imply equality almost everywhere since $f \in L^1_p(\mathbb{R}^d)$.
Let $\phi \in \mathcal{S}(\mathbb{R}^d)$ be an arbitrary test function. By definition of the dual operator, we have:
\[
\langle \mathcal{R}_\sigma^* \mu, \phi \rangle_{\mathbb{R}^d} = \langle \mu, \mathcal{R}_\sigma \phi \rangle_{\mathcal{M}_{p}, C_{0,p}} = \int_{\mathbb{S}^{d-1} \times \mathbb{R}} \mathcal{R}_\sigma \phi(w,b) \, d\mu(w,b).
\]

Using the disintegration lemma, we decompose this global integral over the parameter fibers $d\mu(w,b) = d\mu_w(b) \underline{dw}$:
\[
\langle \mathcal{R}_\sigma^* \mu, \phi \rangle_{\mathbb{R}^d} = \int_{\mathbb{S}^{d-1}} \left( \int_{\mathbb{R}} \mathcal{R}_\sigma \phi(w,b) \, d\mu_w(b) \right) \underline{dw}.
\]

Let us replace $\mathcal{R}_\sigma \phi$ by its explicit spatial slicing on the fiber: $\mathcal{R}_\sigma \phi(w, b) = \int_{\mathbb{R}} \sigma(s-b) \psi_w(s) \, ds$, where $\psi_w(s) = \int_{\langle w, x \rangle = s} \phi(x) \, d\mathcal{H}^{d-1}(x) \in \mathcal{S}(\mathbb{R})$.
The action then writes:
\[
\langle \mathcal{R}_\sigma^* \mu, \phi \rangle_{\mathbb{R}^d} = \int_{\mathbb{S}^{d-1}} \left( \int_{\mathbb{R}} \int_{\mathbb{R}} \sigma(s-b) \psi_w(s) \, ds \, d\mu_w(b) \right) \underline{dw}.
\]

\textbf{1. Absolute convergence and Fubini-Tonelli} \\
To justify the inversion of the inner integrals (exactly as in the proof of Theorem \ref{th:main_th_distrib}), we show that the double integral is absolutely convergent.
Since $\sigma$ has polynomial growth of order $k_\sigma$, we have $|\sigma(s-b)| \le C(1+|s-b|)^{k_\sigma} \le C(1+|s|)^{k_\sigma}(1+|b|)^{k_\sigma}$. Thus:
\[
\int_{\mathbb{R}} \int_{\mathbb{R}} |\sigma(s-b) \psi_w(s)| \, ds \, d|\mu_w|(b) \leq C \left( \int_{\mathbb{R}} (1+|b|)^{k_\sigma} d|\mu_w|(b) \right) \left( \int_{\mathbb{R}} (1+|s|)^{k_\sigma} |\psi_w(s)| \, ds \right).
\]
Since $\psi_w \in \mathcal{S}(\mathbb{R})$, the second integral is finite. For the first one, as $\mu_w \in \mathcal{M}_p(\mathbb{R})$ (weighted growth of order $p$) and $p \ge k_\sigma + 2 > k_\sigma$ by hypothesis, the integral is strictly bounded by $\|\mu_w\|_{\mathrm{TV}, p} < \infty$.

\textbf{2. Isolation of the inner convolution} \\
Fubini's theorem allows us to rearrange the terms to isolate a 1D spatial convolution:
\[
\langle \mathcal{R}_\sigma^* \mu, \phi \rangle_{\mathbb{R}^d} = \int_{\mathbb{S}^{d-1}} \int_{\mathbb{R}} \psi_w(s) \left( \int_{\mathbb{R}} \sigma(s-b) \, d\mu_w(b) \right) ds \, \underline{dw}.
\]
Let us define this inner convolution as $v_w(s) := (\mu_w *_b \sigma)(s)$.
Because for almost every $w \in \mathbb{S}^{d-1}$, $\mu_w \in \mathcal{M}_p$ and $\sigma$ has polynomial growth $k_\sigma$, we can easily bound $|v_w(s)| \le C(1+|s|)^{k_\sigma} \|\mu_w\|_{\mathrm{TV}, p}$. The continuous function $v_w$ is therefore a valid tempered distribution ($v_w \in \mathcal{S}'(\mathbb{R})$).
We can then use the 1D Fourier duality bracket:
\[
\langle \mathcal{R}_\sigma^* \mu, \phi \rangle_{\mathbb{R}^d} = \frac{1}{2\pi} \int_{\mathbb{S}^{d-1}} \langle \widehat{v_w}(\rho), \widehat{\psi_w}(-\rho) \rangle_\rho \, \underline{dw}.
\]

\textbf{3. Determination of the spectrum $\widehat{v_w}(\rho)$} \\
We claim that $\widehat{v_w}(\rho) = \frac{|\mathbb{S}^{d-1}|}{2(2\pi)^{d-1}} |\rho|^{d-1} \widehat{f}(\rho w)$.

Since $\mu_w \in \mathcal{M}_p(\mathbb{R})$ with $p \ge k_\sigma + 2$, Lemma \ref{lem:spectral_convolution_identity} applies and establishes that:
\begin{equation*}
    \widehat{v_w}(\rho) = \widehat{\mu}_w(\rho) \widehat{\sigma}(\rho).
\end{equation*}

By Lemma \ref{lem:identification_spectrale_mu}, we know that for $\underline{dw}$-almost every $w$, $\widehat{\mu}_w = \widehat{H_w}$ in $\mathcal{S}'(\mathbb{R})$. We substitute the explicit definition of $\widehat{H_w}$ from equation \eqref{eq:def_hw}:
\begin{equation*}
    \widehat{v_w}(\rho) = \left( (i\rho)^{\alpha} \frac{|\mathbb{S}^{d-1}|}{2(2\pi)^{d-1}} \frac{|\rho|^{d-1}}{\widehat{g}(\rho)} \widehat{f}(\rho w) \right) \widehat{\sigma}(\rho).
\end{equation*}

Since $\widehat{g}(\rho) = (i\rho)^\alpha \widehat{\sigma}(\rho)$, the terms simplify exactly as in Theorem \ref{th:main_th_distrib}, yielding:
\begin{equation*}
    \widehat{v_w}(\rho) = \frac{|\mathbb{S}^{d-1}|}{2(2\pi)^{d-1}} |\rho|^{d-1} \widehat{f}(\rho w).
\end{equation*}

By Corollary \ref{cor:spectral_simplification}, this distributional simplification across the origin is strictly valid because $(i\rho)^\alpha |\rho|^{d-1} \in C^{\alpha+d-2}(\mathbb{R})$ with $\alpha + d - 2 \ge k_\sigma + 2$ by hypothesis, so that the regularizing factor absorbs all distributional derivatives of $\widehat{\sigma}$ at $\rho=0$.

\textbf{4. Cancellation of singularities and final projection} \\
The projection-slice theorem ensures that $\widehat{\psi_w}(-\rho) = \overline{\widehat{\phi}(\rho w)}$. The integration over the sphere becomes a purely geometric Lebesgue integral:
\[
\langle \mathcal{R}_\sigma^* \mu, \phi \rangle_{\mathbb{R}^d} = \frac{|\mathbb{S}^{d-1}|}{2(2\pi)^d} \int_{\mathbb{S}^{d-1}} \int_{\mathbb{R}} |\rho|^{d-1} \widehat{f}(\rho w) \overline{\widehat{\phi}(\rho w)} \, d\rho \, \underline{dw}.
\]
The change to Cartesian coordinates $\xi = \rho w$ generates the Jacobian $d\rho \, \underline{dw} = \frac{2}{|\mathbb{S}^{d-1}|} \frac{d\xi}{|\xi|^{d-1}}$.
The radial singularity cancels out, yielding:
\[
\langle \mathcal{R}_\sigma^* \mu, \phi \rangle_{\mathbb{R}^d} = \frac{1}{(2\pi)^d} \int_{\mathbb{R}^d} \widehat{f}(\xi) \overline{\widehat{\phi}(\xi)} \, d\xi.
\]
By Plancherel's identity, this is strictly equal to $\langle f, \phi \rangle_{\mathbb{R}^d}$.
Since this equality holds for any test function $\phi \in \mathcal{S}(\mathbb{R}^d)$ and $f \in L^1_p(\mathbb{R}^d)$, we rigorously conclude that $f = \mathcal{R}_\sigma^* \mu$.

\end{proof}

\begin{remark}
The explicit reliance on the uniform spherical Lebesgue measure $\underline{dw}$ in Definition \ref{def:def_Hw_Mp} might initially appear arbitrary. However, the final step of the preceding proof highlights its structural necessity: without the integration with respect to $\underline{dw}$, we could not identify the polar integral
\[
\frac{|\mathbb{S}^{d-1}|}{2(2\pi)^d} \int_{\mathbb{S}^{d-1}} \int_{\mathbb{R}} |\rho|^{d-1} \widehat{f}(\rho w) \overline{\widehat{\phi}(\rho w)} \, d\rho \, \underline{dw}
\]
with its Cartesian counterpart
\[
\frac{1}{(2\pi)^d} \int_{\mathbb{R}^d} \widehat{f}(\xi) \overline{\widehat{\phi}(\xi)} \, d\xi,
\]
which is strictly required to conclude the proof via Plancherel's identity.
\end{remark}

We next demonstrate that spectral Barron spaces satisfy the representation framework of Theorem \ref{th:carac_gen_barron_mp}, adapting the harmonic perspective of \cite{schavemaker2025barron} to our weighted Radon setting.

\begin{lemma}[1D Sobolev--Fourier Embedding for Weighted $L^1$ Spaces]\label{lem:sobolev_fourier_1d}
Let $G \in W^{m,1}(\mathbb{R})$ for an integer $m \ge 2$, and let $h = \mathcal{F}^{-1}\{G\} \in \mathcal{S}'(\mathbb{R})$ be its inverse Fourier transform in the sense of distributions. Then $h$ is a bounded continuous function that belongs to $L^1_p(\mathbb{R})$ for any $0 \le p \le m-2$, and there exists a constant $C_{m,p} > 0$ depending only on $m$ and $p$ such that:
\[
\|h\|_{L^1_p(\mathbb{R})} = \int_{\mathbb{R}} (1+|b|)^p |h(b)| \, db \le C_{m,p} \|G\|_{W^{m,1}(\mathbb{R})}.
\]
\end{lemma}

\begin{proof}
Since $G \in W^{m,1}(\mathbb{R})$, both $G$ and its $m$-th weak derivative $G^{(m)}$ belong to $L^1(\mathbb{R})$. Because $G \in L^1(\mathbb{R})$, the Riemann--Lebesgue lemma guarantees that its inverse Fourier transform $h(b) = \mathcal{F}^{-1}\{G\}(b) = \frac{1}{2\pi} \int_{\mathbb{R}} G(\rho) e^{i b \rho} \, d\rho$ is a bounded continuous function on $\mathbb{R}$ satisfying:
\[
\|h\|_{L^\infty(\mathbb{R})} \le \frac{1}{2\pi} \|G\|_{L^1(\mathbb{R})}.
\]
Furthermore, by standard properties of the Fourier transform on tempered distributions, the $m$-th weak derivative $G^{(m)} \in L^1(\mathbb{R})$ satisfies $\mathcal{F}^{-1}\{G^{(m)}\}(b) = (-ib)^m h(b)$ in $\mathcal{S}'(\mathbb{R})$. Because $G^{(m)} \in L^1(\mathbb{R})$, its inverse Fourier transform is also a continuous bounded function:
\[
|b|^m |h(b)| \le \frac{1}{2\pi} \|G^{(m)}\|_{L^1(\mathbb{R})} \quad \forall b \in \mathbb{R}.
\]
Combining these two pointwise bounds, we obtain:
\[
(1+|b|^m) |h(b)| \le \frac{1}{2\pi} \left( \|G\|_{L^1(\mathbb{R})} + \|G^{(m)}\|_{L^1(\mathbb{R})} \right) \le \frac{1}{2\pi} \|G\|_{W^{m,1}(\mathbb{R})}.
\]
Finally, for any $0 \le p \le m-2$, the weight ratio $b \mapsto \frac{(1+|b|)^p}{1+|b|^m}$ is strictly integrable over $\mathbb{R}$ since $m - p \ge 2$. We can thus integrate:
\[
\int_{\mathbb{R}} (1+|b|)^p |h(b)| \, db = \int_{\mathbb{R}} \frac{(1+|b|)^p}{1+|b|^m} \Big( (1+|b|^m) |h(b)| \Big) \, db \le C_{m,p} \|G\|_{W^{m,1}(\mathbb{R})},
\]
where $C_{m,p} = \frac{1}{2\pi} \int_{\mathbb{R}} \frac{(1+|b|)^p}{1+|b|^m} \, db < \infty$. This proves that $h \in L^1_p(\mathbb{R})$ and establishes the desired bound.
\end{proof}

\begin{theorem}[Sobolev--Spectral Barron Embedding]\label{th:spectral_barron_embedding_fixed}
Let $f \in L^1_p(\mathbb{R}^d)$ with $p \ge 0$, and set $m :=  p  + 2$. For almost every $w \in \mathbb{S}^{d-1}$, let $\widehat{H}_w$ denote the 1D fiber analysis spectrum defined globally on $\mathbb{R}$ by:
\[
\widehat{H}_w(\rho) = c_d \frac{(i\rho)^\alpha |\rho|^{d-1}}{\widehat{g}(\rho)} \widehat{f}(\rho w), \quad \text{where } c_d = \frac{|\mathbb{S}^{d-1}|}{2(2\pi)^{d-1}}.
\]
Assume that:
\begin{enumerate}
    \item For almost every $w \in \mathbb{S}^{d-1}$, $\widehat{H}_w \in W^{m,1}(\mathbb{R} \setminus \{0\})$, and satisfies
    \[
    \int_{\mathbb{S}^{d-1}} \|\widehat{H}_w\|_{W^{m,1}(\mathbb{R} \setminus \{0\})} \, \underline{dw} < \infty.
    \]
    \item \textbf{(Trace Compatibility at $\rho = 0$):} For almost every $w \in \mathbb{S}^{d-1}$, the one-sided traces of $\widehat{H}_w$ at the origin satisfy the matching conditions:
    \[
    \lim_{\rho \to 0^+} \widehat{H}_w^{(k)}(\rho) = \lim_{\rho \to 0^-} \widehat{H}_w^{(k)}(\rho) \quad \text{for all } k = 0, 1, \ldots, m-1.
    \]
\end{enumerate}
Then the filtered analysis distribution $\mu = (-\Delta_b)^\alpha \mathcal{T}_g *_b \mathcal{R}_\sigma f$ extends uniquely to a finite weighted Radon measure in $\mathcal{M}_p(\mathbb{S}^{d-1} \times \mathbb{R})$, and
\[
\|\mu\|_{\mathrm{TV}, p} \le C \int_{\mathbb{S}^{d-1}} \|\widehat{H}_w\|_{W^{m,1}(\mathbb{R})} \, \underline{dw}.
\]
\end{theorem}

\begin{remark}[Trace Compatibility in the High-Dimensional Regime]\label{rmk:trace_high_dim}
At first glance, the trace compatibility requirement at the frequency origin (Condition 2 of Theorem~\ref{th:spectral_barron_embedding_fixed}) might appear to impose a restrictive structural constraint on the target function. However, in the high-dimensional regime relevant to modern machine learning, this condition is automatically satisfied.

To see why, observe that the 1D fiber analysis spectrum is modulated by the frequency quotient
\[
\widehat{H}_w(\rho) = c_d \frac{(i\rho)^\alpha |\rho|^{d-1}}{\widehat{g}(\rho)} \widehat{f}(\rho w).
\]
By Hypothesis~\ref{hyp:non_zero_g}, $\widehat{g}$ does not vanish at the origin. Therefore, the quotient $\frac{(i\rho)^\alpha |\rho|^{d-1}}{\widehat{g}(\rho)}$ inherently retains an algebraic root of order $\alpha + d - 1 \ge 1$ at the origin.

By Leibniz's differentiation rule, any one-sided derivative of order $k$ of this quotient will retain at least one strictly positive power of $|\rho|$ as long as the differentiation order is strictly below its algebraic degree, namely $k < d + \alpha - 1$. Consequently, both left- and right-sided limits vanish identically at the origin:
\[
\lim_{\rho \to 0^+} \widehat{H}_w^{(k)}(\rho) = \lim_{\rho \to 0^-} \widehat{H}_w^{(k)}(\rho) = 0 \quad \text{for all } 0 \le k < d + \alpha - 1.
\]
Since Condition 2 requires trace matching up to order $k = m - 1$, the compatibility condition is \textbf{unconditionally satisfied} ($0 = 0$) whenever the Sobolev differentiation order satisfies:
\[
m \le d + \alpha - 1.
\]
In standard deep learning settings, the ambient dimension $d$ is typically very large ($d \gg 1$), whereas the polynomial weight growth order $p$ (where $m = p + 2$) remains moderate ($m \ll d$). In this high-dimensional regime, the spherical geometric factor $|\rho|^{d-1}$ naturally crushes any potential origin singularity. Therefore, Theorem~\ref{th:spectral_barron_embedding_fixed} applies without any practical restriction on the traces at $\rho = 0$.
\end{remark}

\begin{proof}
For almost every $w \in \mathbb{S}^{d-1}$, the 1D Fourier transform $\widehat{H}_w(\rho)$ of the analysis fiber distribution belongs to $W^{m,1}(\mathbb{R} \setminus \{0\})$ by Assumption 1.

To establish that $\widehat{H}_w \in W^{m,1}(\mathbb{R})$ across the origin $\rho = 0$, we verify the continuity of its weak derivatives up to order $m-1$. For any $0 \le k \le m-1$, Assumption 2 guarantees that:
\[
\lim_{\rho \to 0^+} \widehat{H}_w^{(k)}(\rho) = \lim_{\rho \to 0^-} \widehat{H}_w^{(k)}(\rho).
\]
Consequently, $\widehat{H}_w$ possesses no jump discontinuities at $\rho = 0$ up to order $m-1$, guaranteeing that its weak derivatives $\widehat{H}_w^{(k)}$ up to order $m$ contain no singular Dirac delta measures supported at $0$. This confirms that $\widehat{H}_w \in W^{m,1}(\mathbb{R})$, with:
\[
\|\widehat{H}_w\|_{W^{m,1}(\mathbb{R})} = \|\widehat{H}_w\|_{W^{m,1}(\mathbb{R} \setminus \{0\})}.
\]

Applying Lemma \ref{lem:sobolev_fourier_1d} to $G = \widehat{H}_w \in W^{m,1}(\mathbb{R})$ with $m = p + 2$, we obtain that the inverse Fourier transform $H_w = \mathcal{F}^{-1}\{\widehat{H}_w\}$ is a regular function belonging to $L^1_p(\mathbb{R})$ for $\underline{dw}$-almost every $w \in \mathbb{S}^{d-1}$, satisfying the fiber estimate:
\[
\|H_w\|_{L^1_p(\mathbb{R})} = \int_{\mathbb{R}} (1+|b|)^p |H_w(b)| \, db \le C_{m,p} \|\widehat{H}_w\|_{W^{m,1}(\mathbb{R})}.
\]

We now define the Borel measure $\mu$ on the cylinder $\mathbb{S}^{d-1} \times \mathbb{R}$ by:
\[
d\mu(w, b) := H_w(b) \, db \, \underline{dw}.
\]
Its weighted total variation is evaluated using the Fubini--Tonelli theorem:
\[
\|\mu\|_{\mathrm{TV}, p} = \int_{\mathbb{S}^{d-1} \times \mathbb{R}} (1+|b|)^p \, d|\mu|(w, b) = \int_{\mathbb{S}^{d-1}} \|H_w\|_{L^1_p(\mathbb{R})} \, \underline{dw} \le C_{m,p} \int_{\mathbb{S}^{d-1}} \|\widehat{H}_w\|_{W^{m,1}(\mathbb{R})} \, \underline{dw} < \infty.
\]
Thus, $\mu \in \mathcal{M}_p(\mathbb{S}^{d-1} \times \mathbb{R})$. Moreover, for any joint test function $\phi \in \mathcal{S}(\mathbb{S}^{d-1} \times \mathbb{R})$:
\[
\int_{\mathbb{S}^{d-1} \times \mathbb{R}} \phi(w, b) \, d\mu(w, b) = \int_{\mathbb{S}^{d-1}} \left( \int_{\mathbb{R}} H_w(b) \phi(w, b) \, db \right) \underline{dw} = \int_{\mathbb{S}^{d-1}} \langle H_w, \phi(w, \cdot) \rangle_{\mathbb{R}} \, \underline{dw}.
\]
This precisely verifies the condition of Definition \ref{def:def_Hw_Mp}, confirming that the filtered analysis distribution $(-\Delta_b)^\alpha \mathcal{T}_g *_b \mathcal{R}_\sigma f$ coincides with the weighted Radon measure $\mu \in \mathcal{M}_p(\mathbb{S}^{d-1} \times \mathbb{R})$, concluding the proof.
\end{proof}

\begin{remark}[Connection to the Classical Spectral Barron Space and the Unbounded Domain]
For a ReLU activation ($\alpha=2$ and $\widehat{g}(\rho) = 1$), the 1D fiber analysis spectrum simplifies on $\mathbb{R} \setminus \{0\}$ to $\widehat{H}_w(\rho) = - c_d \rho^2 |\rho|^{d-1} \widehat{f}(\rho w)$, where $c_d = \frac{|\mathbb{S}^{d-1}|}{2(2\pi)^{d-1}}$. Integrating the $L^1(\mathbb{R})$ norm of $\widehat{H}_w$ over the unit sphere $\mathbb{S}^{d-1}$ directly recovers the classical spectral Barron condition up to a dimensional constant:
\[
\int_{\mathbb{S}^{d-1}} \|\widehat{H}_w\|_{L^1(\mathbb{R})} \, \underline{dw} = \frac{1}{(2\pi)^{d-1}} \int_{\mathbb{R}^d} |\xi|^2 |\widehat{f}(\xi)| \, d\xi < \infty.
\]
While the classical Barron space requires only $L^1$ integrability of $\widehat{H}_w$, our embedding theorem imposes the stronger Sobolev condition $\widehat{H}_w \in W^{m,1}(\mathbb{R} \setminus \{0\})$ with $m =  p + 2 \ge 2$, along with trace matching conditions at $\rho=0$. 

This structural difference is the direct consequence of formulating the representation globally over the unbounded domain $\mathbb{R}^d$:
\begin{itemize}
    \item \textbf{Compact Domains ($L^\infty$ Control):} On a compact domain $\Omega \subset \mathbb{R}^d$, the neural network outputs remain bounded as long as the bias parameter density $H_w(b)$ is bounded in $L^\infty(\mathbb{R})$. The classical condition $\int_{\mathbb{R}^d} |\xi|^2 |\widehat{f}(\xi)| d\xi < \infty$ delivers precisely this $L^\infty$ bound on $H_w$ via the standard $L^1 \to L^\infty$ Fourier embedding $\mathcal{F}^{-1}: L^1(\mathbb{R}) \to L^\infty(\mathbb{R})$.
    
    \item \textbf{Unbounded Domain $\mathbb{R}^d$ ($L^1_p$ Control and Fourier Duality):} To extend the representation isometry globally to $\mathbb{R}^d$, the parameter density $H_w(b)$ must be strictly integrable along the bias fibers, i.e., $H_w \in L^1_p(\mathbb{R})$. By fundamental Fourier duality, \textit{decay at infinity in the spatial bias domain $b \in \mathbb{R}$ is strictly equivalent to Sobolev smoothness in the frequency domain $\rho \in \mathbb{R}$}. 
\end{itemize}
Consequently, the $m$ weak derivatives in $W^{m,1}$ and the trace compatibility conditions at $\rho=0$ represent the precise harmonic analysis cost required to guarantee that $H_w \in L^1_p(\mathbb{R})$ and prevent heavy-tailed, non-integrable bias distributions, ensuring that $\mu$ remains a finite Radon measure on the full parameter cylinder $\mathbb{S}^{d-1} \times \mathbb{R}$.
\end{remark}

\section{Even-Odd Decomposition of the Kernel and Existence of an Optimal Linear Operator}\label{sec:optimal_operator}

\begin{definition}[Generalized Barron Space with Asymptotic Decay]\label{def:generalized_barron_space}
Let $p \ge k_\sigma + \alpha + 2$. The generalized Barron space $\mathcal{B} \subset L^1_p(\mathbb{R}^d)$ is defined as the space of spatially decaying functions that admit a continuous neural representation by a weighted Radon measure whose angular marginal is absolutely continuous with respect to $\underline{dw}$:
\begin{equation}\label{eq:def_barron_space}
    \mathcal{B} := \left\{ f \in L^1_p(\mathbb{R}^d) \;\middle|\; \exists \mu = \mu_w \underline{dw} \in \mathcal{M}_{p}(\mathbb{S}^{d-1} \times \mathbb{R}) \text{ such that } \mathcal{R}_\sigma^* \mu = f \right\}.
\end{equation}
It is naturally equipped with the quotient variation norm:
\begin{equation}\label{eq:barron_quotient_norm}
    \|f\|_{\mathcal{B}} := \inf \left\{ \|\mu\|_{\mathrm{TV}, p} \;\middle|\; \mu = \mu_w \underline{dw} \in \mathcal{M}_{p}(\mathbb{S}^{d-1} \times \mathbb{R}), \ \mathcal{R}_\sigma^* \mu = f \right\}.
\end{equation}
\end{definition}

In our purely continuous framework on $L^1_p(\mathbb{R}^d)$, one might naturally seek an optimal \textbf{linear} analysis operator $P : \mathcal{B} \to \mathcal{M}_p(\mathbb{S}^{d-1} \times \mathbb{R})$ that associates to each target function $f \in \mathcal{B}$ a valid representing measure, thereby strictly preserving the topology (i.e., finding the exact measure that achieves the infimum of the quotient norm). 

\begin{lemma}[Spectral Characterization of Representability]\label{lem:spectral_characterization_f}
Assume $d \ge 1$. Let $f \in L^1_p(\mathbb{R}^d)$ with $p \in \mathbb{N}$, $p \ge k_\sigma + \alpha + 2$, and let $\mu = \mu_w \underline{dw} \in \mathcal{M}_p(\mathbb{S}^{d-1} \times \mathbb{R})$. Suppose the regularized spectrum $\widehat{g}(\rho) = (i\rho)^\alpha \widehat{\sigma}(\rho)$ satisfies Hypothesis \ref{hyp:non_zero_g} and is even: $\widehat{g}(-\rho) = \widehat{g}(\rho)$. Let $\epsilon_\sigma = (-1)^\alpha$.

\begin{enumerate}
    \item \textbf{Necessity:} If the measure $\mu$ represents the function $f$ in the sense of distributions (i.e., $\mathcal{R}_\sigma^* \mu = f$), then its partial Fourier transform $\widehat{\mu}_w = \mathcal{F}_b\{\mu_w\} \in \mathcal{S}'(\mathbb{R})$ satisfies for almost every $w \in \mathbb{S}^{d-1}$:
    \begin{equation}\label{eq:spectral_kernel_cond}
        \widehat{\mu}_w(\rho) + \epsilon_\sigma \widehat{\mu}_{-w}(-\rho) = (i\rho)^\alpha \frac{|\mathbb{S}^{d-1}|}{(2\pi)^{d-1}} \frac{|\rho|^{d-1}}{\widehat{g}(\rho)} \widehat{f}(\rho w)
    \end{equation}
    in the sense of $\mathcal{S}'(\mathbb{R})$.

    \item \textbf{Sufficiency:} Conversely, if $\widehat{\mu}_w \in \mathcal{S}'(\mathbb{R})$ satisfies \eqref{eq:spectral_kernel_cond} and additionally satisfies the origin-vanishing moments condition for almost every $w \in \mathbb{S}^{d-1}$:
    \begin{equation}\label{eq:origin_vanishing_moments}
        \widehat{\mu}_w^{(k)}(0) = 0 \quad \text{for all } 0 \le k < \alpha,
    \end{equation}
    then $\mathcal{R}_\sigma^* \mu = f$ in the sense of distributions.
\end{enumerate}
\end{lemma}

\begin{proof}
We divide the proof into two parts, establishing the necessity and sufficiency of the spectral condition.
\paragraph{Necessity ($\mathcal{R}_\sigma^{*} \mu = f \implies$ Spectral Condition):}

By definition, the measure $\mu$ represents $f$ if and only if  $\langle \mathcal{R}_\sigma^* \mu, \phi \rangle_{\mathbb{R}^d} = \langle f, \phi \rangle_{\mathbb{R}^d}$ for any spatial test function $\phi \in \mathcal{S}(\mathbb{R}^d)$.
By computation from the proof of Theorem \ref{th:carac_gen_barron_mp}:
\begin{equation*}
    \langle \mathcal{R}_\sigma^* \mu, \phi \rangle_{\mathbb{R}^d} = \frac{1}{2\pi} \int_{\mathbb{S}^{d-1}} \langle \widehat{v_w}(\rho), \widehat{\phi}(-\rho w) \rangle_\rho \, \underline{dw},
\end{equation*}
where the spatial function on the fiber $v_w(s) := (\mu_w *_b \sigma)(s)$ has the spectral identity $\widehat{v_w}(\rho) = \widehat{\mu}_w(\rho) \widehat{\sigma}(\rho)$ directly according to Lemma \ref{lem:spectral_convolution_identity}.

Simultaneously, we express the target action $\langle f, \phi \rangle_{\mathbb{R}^d}$ in the frequency domain using polar coordinates (setting $\xi = \rho w$ with $d\xi = |\mathbb{S}^{d-1}| \rho^{d-1} d\rho \, \underline{dw}$ for positive radii $\rho \in \mathbb{R}^+$):
\begin{align*}
    \langle f, \phi \rangle_{\mathbb{R}^d} &= \frac{1}{(2 \pi)^d} \int_{\mathbb{R}^d} \widehat{f}(\xi) \widehat{\phi}(-\xi) \, d\xi \\
    &= \frac{|\mathbb{S}^{d-1}|}{(2 \pi)^d} \int_{\mathbb{S}^{d-1}} \int_{\mathbb{R}^+} \rho^{d-1} \widehat{f}(\rho w) \widehat{\phi}(-\rho w) \, d\rho \, \underline{dw} \\
    &= \frac{|\mathbb{S}^{d-1}|}{2(2 \pi)^d} \int_{\mathbb{S}^{d-1}} \int_{\mathbb{R}} |\rho|^{d-1} \widehat{f}(\rho w) \widehat{\phi}(-\rho w) \, d\rho \, \underline{dw}.
\end{align*}

Equating the two actions and factoring out the constants yields:
\begin{equation*}
    \frac{1}{2\pi} \int_{\mathbb{S}^{d-1}} \left\langle \widehat{\mu}_w(\rho) \widehat{\sigma}(\rho) - \frac{|\mathbb{S}^{d-1}|}{2(2\pi)^{d-1}} |\rho|^{d-1} \widehat{f}(\rho w), \widehat{\phi}(-\rho w) \right\rangle_\rho \, \underline{dw} = 0.
\end{equation*}

By exploiting the central involution $(w, \rho) \mapsto (-w, -\rho)$, which leaves both the normalized spherical measure $\underline{dw}$ and the Lebesgue measure $d\rho$ invariant, we symmetrize the integral over the parameter cylinder:
\begin{equation}\label{eq:equation_nuw_f}
    \frac{1}{4\pi} \int_{\mathbb{S}^{d-1}} \left\langle \widehat{\mu}_w(\rho) \widehat{\sigma}(\rho) + \widehat{\mu}_{-w}(-\rho) \widehat{\sigma}(-\rho) - \frac{|\mathbb{S}^{d-1}|}{(2\pi)^{d-1}} |\rho|^{d-1} \widehat{f}(\rho w), \widehat{\phi}(-\rho w) \right\rangle_\rho \, \underline{dw} = 0.
\end{equation}

Let us define the distribution on the fibers:
\begin{equation*}
    G_w(\rho) := \widehat{\mu}_w(\rho) \widehat{\sigma}(\rho) + \widehat{\mu}_{-w}(-\rho) \widehat{\sigma}(-\rho) - \frac{|\mathbb{S}^{d-1}|}{(2\pi)^{d-1}} |\rho|^{d-1} \widehat{f}(\rho w).
\end{equation*}
Notice that $G_w$ possesses an inherent central inversion symmetry. Evaluating $G_{-w}(-\rho)$ yields:
\begin{align*}
    G_{-w}(-\rho) &= \widehat{\mu}_{-w}(-\rho) \widehat{\sigma}(-\rho) + \widehat{\mu}_{w}(\rho) \widehat{\sigma}(\rho) - \frac{|\mathbb{S}^{d-1}|}{(2\pi)^{d-1}} |-\rho|^{d-1} \widehat{f}(-\rho (-w)) \\
    &= \widehat{\mu}_w(\rho) \widehat{\sigma}(\rho) + \widehat{\mu}_{-w}(-\rho) \widehat{\sigma}(-\rho) - \frac{|\mathbb{S}^{d-1}|}{(2\pi)^{d-1}} |\rho|^{d-1} \widehat{f}(\rho w) = G_w(\rho).
\end{align*}

To rigorously isolate the fiber distributions, we introduce a global frequency test function $\Phi \in \mathcal{S}(\mathbb{R}^d)$ defined by separating the radial and angular components:
\begin{equation*}
    \Phi(\xi) = 
    \begin{cases} 
    A\left(\frac{\xi}{|\xi|}\right) \psi(|\xi|) & \text{if } \xi \neq 0, \\
    0 & \text{if } \xi = 0,
    \end{cases}
\end{equation*}
where $A \in C^\infty(\mathbb{S}^{d-1})$ is an arbitrary smooth angular function, and $\psi \in \mathcal{D}(\mathbb{R})$ is a symmetric radial test function (i.e., $\psi(\rho) = \psi(-\rho)$) whose support strictly excludes the origin ($\rho = 0$). Setting $\widehat{\phi}(-\xi) = \Phi(\xi)$, its evaluation strictly yields a symmetrized decoupled function on the fibers ($\Phi(\rho w) = A(w) \psi(\rho)$ for $\rho > 0$, and $A(-w) \psi(\rho)$ for $\rho < 0$).

Substituting this into our integral \eqref{eq:equation_nuw_f}, we split the integration domain over the radial variable:
\begin{equation*}
    \int_{\mathbb{S}^{d-1}} \left( \int_0^\infty G_w(\rho) A(w) \psi(\rho) \, d\rho + \int_{-\infty}^0 G_w(\rho) A(-w) \psi(-\rho) \, d\rho \right) \underline{dw} = 0
\end{equation*}
where the integral notation on $\rho$ is allowed since $0$ (the only potential singularity of $\widehat{\sigma}$) is outside the support of $\psi$. 

In the second term, we apply the change of variable $\rho \mapsto -\rho$ (with $\psi(-\rho) = \psi(\rho)$) and use the invariance of $\underline{dw}$ under the antipodal mapping $w \mapsto -w$. Moreover, recombining the two terms and utilizing the symmetry property $G_w(\rho) = G_{-w}(-\rho)$, we obtain:
\begin{equation*}
    \int_{\mathbb{S}^{d-1}} A(w) \left( \int_0^\infty 2 G_w(\rho) \psi(\rho) \, d\rho \right) \underline{dw} = 0.
\end{equation*}

Since this holds for any smooth angular function $A$, the inner integral must vanish for almost every $w \in \mathbb{S}^{d-1}$:
\begin{equation*}
    \int_0^\infty G_w(\rho) \psi(\rho) \, d\rho = 0.
\end{equation*}

Because $\psi$ acts as an arbitrary test function on the positive half-line, we conclude that $G_w = 0$ on $\mathbb{R}^{+\star}$. By the symmetry $G_w(\rho) = G_{-w}(-\rho)$, it immediately follows that $G_w = 0$ on $\mathbb{R}^{-\star}$ as well. Thus, we have rigorously deduced that:
\begin{equation*}
    G_w = 0 \quad \text{in } \mathcal{D}'(\mathbb{R} \setminus \{0\}).
\end{equation*}

To introduce the regularized spectrum $\widehat{g}(\rho) = (i\rho)^\alpha \widehat{\sigma}(\rho)$, we multiply the entire equation by $(i\rho)^\alpha$:
\begin{equation*}
    \widehat{\mu}_w(\rho) \widehat{g}(\rho) + \widehat{\mu}_{-w}(-\rho) (i\rho)^\alpha \widehat{\sigma}(-\rho) = (i\rho)^\alpha \frac{|\mathbb{S}^{d-1}|}{(2\pi)^{d-1}} |\rho|^{d-1} \widehat{f}(\rho w) \quad \text{in } \mathcal{D}'(\mathbb{R} \setminus \{0\}).
\end{equation*}

By noting that $(i\rho)^\alpha = (-1)^\alpha (-i\rho)^\alpha$, the second term contains $(-i\rho)^\alpha \widehat{\sigma}(-\rho) = \widehat{g}(-\rho)$. Thus:
\begin{equation*}
    \widehat{\mu}_w(\rho) \widehat{g}(\rho) + (-1)^\alpha \widehat{\mu}_{-w}(-\rho) \widehat{g}(-\rho) = (i\rho)^\alpha \frac{|\mathbb{S}^{d-1}|}{(2\pi)^{d-1}} |\rho|^{d-1} \widehat{f}(\rho w) \quad \text{in } \mathcal{D}'(\mathbb{R} \setminus \{0\}).
\end{equation*}

Given that $\widehat{g}$ is even ($\widehat{g}(-\rho) = \widehat{g}(\rho)$), we can explicitly factor it out. Defining $\epsilon_\sigma = (-1)^\alpha$, we divide by $\widehat{g}(\rho)$ to strictly obtain:
\begin{equation*}
    \widehat{\mu}_w(\rho) + \epsilon_\sigma \widehat{\mu}_{-w}(-\rho) = (i\rho)^\alpha \frac{|\mathbb{S}^{d-1}|}{(2\pi)^{d-1}} \frac{|\rho|^{d-1}}{\widehat{g}(\rho)} \widehat{f}(\rho w) \quad \text{in } \mathcal{D}'(\mathbb{R} \setminus \{0\}).
\end{equation*}

Because $\mu_w \in \mathcal{M}_p(\mathbb{R})$ is a finite Radon measure for almost every $w \in \mathbb{S}^{d-1}$, its Fourier transform $\widehat{\mu}_w$ is a continuous function. Moreover, since $f \in L^1_p(\mathbb{R}^d)$ and $\widehat{g}(0) \neq 0$ by Hypothesis~\ref{hyp:non_zero_g}, the right-hand side is also a continuous function on $\mathbb{R}$ (the factor $(i\rho)^\alpha |\rho|^{d-1}$ vanishing at $\rho=0$). Since two continuous functions that coincide on $\mathbb{R} \setminus \{0\}$ must coincide on all of $\mathbb{R}$ (precluding any singular Dirac distributions at the origin), the equality rigorously extends to the whole space:
\begin{equation*}
    \widehat{\mu}_w(\rho) + \epsilon_\sigma \widehat{\mu}_{-w}(-\rho) = (i\rho)^\alpha \frac{|\mathbb{S}^{d-1}|}{(2\pi)^{d-1}} \frac{|\rho|^{d-1}}{\widehat{g}(\rho)} \widehat{f}(\rho w) \quad \text{in } \mathcal{S}'(\mathbb{R}).
\end{equation*}

\paragraph{Sufficiency (Spectral Condition $\implies \mathcal{R}_\sigma^* \mu = f$):}
Assume that for almost every $w \in \mathbb{S}^{d-1}$, both the spectral equation \eqref{eq:spectral_kernel_cond} and the origin-vanishing moments condition \eqref{eq:origin_vanishing_moments} hold:
\begin{equation}\label{eq:suff_premise}
    \widehat{\mu}_w(\rho) + \epsilon_\sigma \widehat{\mu}_{-w}(-\rho) = (i\rho)^\alpha \frac{|\mathbb{S}^{d-1}|}{(2\pi)^{d-1}} \frac{|\rho|^{d-1}}{\widehat{g}(\rho)} \widehat{f}(\rho w) \quad \text{in } \mathcal{S}'(\mathbb{R}),
\end{equation}
and $\widehat{\mu}_w^{(k)}(0) = 0$ for all $0 \le k < \alpha$.

We first analyze the restriction to the open punctured line $\mathbb{R} \setminus \{0\}$. On this domain, the regularized spectrum $\widehat{g}(\rho) = (i\rho)^\alpha \widehat{\sigma}(\rho)$ is smooth and non-zero everywhere by hypothesis. Therefore, $\widehat{\sigma}(\rho) = (i\rho)^{-\alpha} \widehat{g}(\rho)$ is a well-defined, smooth function on $\mathbb{R} \setminus \{0\}$, and its parity satisfies pointwise for all $\rho \neq 0$:
\begin{equation*}
    \widehat{\sigma}(-\rho) = \frac{\widehat{g}(-\rho)}{(-i\rho)^\alpha} = \frac{\widehat{g}(\rho)}{(-1)^\alpha (i\rho)^\alpha} = (-1)^{-\alpha} \widehat{\sigma}(\rho) = \epsilon_\sigma \widehat{\sigma}(\rho).
\end{equation*}

Since $\widehat{\sigma}(\rho)$ is smooth on $\mathbb{R} \setminus \{0\}$, multiplying equation \eqref{eq:suff_premise} by $\widehat{\sigma}(\rho)$ is strictly legal and associative in $\mathcal{D}'(\mathbb{R} \setminus \{0\}).$ Using the identity $(i\rho)^\alpha \widehat{\sigma}(\rho) = \widehat{g}(\rho)$ and injecting the parity relation $\epsilon_\sigma \widehat{\sigma}(\rho) = \widehat{\sigma}(-\rho)$, the non-zero denominator cancels out perfectly on $\mathbb{R} \setminus \{0\}$:
\begin{equation}\label{eq:suff_punctured}
    \widehat{\mu}_w(\rho) \widehat{\sigma}(\rho) + \widehat{\mu}_{-w}(-\rho) \widehat{\sigma}(-\rho) = \frac{|\mathbb{S}^{d-1}|}{(2\pi)^{d-1}} |\rho|^{d-1} \widehat{f}(\rho w) \quad \text{in } \mathcal{D}'(\mathbb{R} \setminus \{0\}).
\end{equation}

We now extend this distributional equality from $\mathcal{D}'(\mathbb{R} \setminus \{0\})$ to the full tempered distribution space $\mathcal{S}'(\mathbb{R})$. Because two distributions in $\mathcal{S}'(\mathbb{R})$ that coincide on $\mathbb{R} \setminus \{0\}$ can differ at most by a finite linear combination of Dirac deltas supported at the origin, we must verify that no singular point-support terms $\sum_{k=0}^{N} c_k(w) \delta^{(k)}(\rho)$ appear when extending the equality to $\mathcal{S}'(\mathbb{R})$. 

By Lemma \ref{lem:disintegration_mu}, since $\mu_w \in \mathcal{M}_p(\mathbb{R})$ with $p \ge k_\sigma + \alpha + 2$, its Fourier transform $\widehat{\mu}_w$ is a $C^p$ function on $\mathbb{R}$. Crucially, by condition \eqref{eq:origin_vanishing_moments}, $\widehat{\mu}_w(\rho)$ vanishes at the origin up to order $\alpha - 1$. By Taylor's theorem, we can factor out the root of multiplicity $\alpha$ explicitly: there exists a function $h_w \in C^{p-\alpha}(\mathbb{R}) \subset C^{k_\sigma + 2}(\mathbb{R})$ such that
\begin{equation*}
    \widehat{\mu}_w(\rho) = (i\rho)^\alpha h_w(\rho) \quad \text{for all } \rho \in \mathbb{R}.
\end{equation*}
Since $\widehat{\sigma} \in \mathcal{S}'(\mathbb{R})$ has order at most $k_\sigma + 2$ and $h_w \in C^{k_\sigma + 2}(\mathbb{R})$, the distributional product $\widehat{\mu}_w(\rho) \widehat{\sigma}(\rho)$ in $\mathcal{S}'(\mathbb{R})$ simplifies globally across the origin thanks to Theorem \ref{th:associativity_distrib}:
\begin{equation*}
    \widehat{\mu}_w(\rho) \widehat{\sigma}(\rho) = h_w(\rho) \Big( (i\rho)^\alpha \widehat{\sigma}(\rho) \Big) = h_w(\rho) \widehat{g}(\rho).
\end{equation*}
Because $h_w$ is continuous and $\widehat{g} \in C^\infty(\mathbb{R})$, the product $\widehat{\mu}_w(\rho) \widehat{\sigma}(\rho)$ is a well-defined regular continuous function on the entire real line $\mathbb{R}$. Similarly, $\widehat{\mu}_{-w}(-\rho) \widehat{\sigma}(-\rho)$ is continuous on $\mathbb{R}$, and since $f \in L^1_p(\mathbb{R}^d)$, the right-hand side of \eqref{eq:suff_punctured} is also a continuous function on all of $\mathbb{R}$.

Since two continuous functions on $\mathbb{R}$ that coincide on $\mathbb{R} \setminus \{0\}$ must coincide everywhere without any singular Dirac delta discrepancy (i.e., $c_k(w) = 0$ for all $k$), the equality holds globally in $\mathcal{S}'(\mathbb{R})$:
\begin{equation*}
    \widehat{\mu}_w(\rho) \widehat{\sigma}(\rho) + \widehat{\mu}_{-w}(-\rho) \widehat{\sigma}(-\rho) = \frac{|\mathbb{S}^{d-1}|}{(2\pi)^{d-1}} |\rho|^{d-1} \widehat{f}(\rho w) \quad \text{in } \mathcal{S}'(\mathbb{R}).
\end{equation*}

Testing this identity against $\widehat{\phi}(-\rho w)$ for an arbitrary spatial test function $\phi \in \mathcal{S}(\mathbb{R}^d)$ and integrating over the parameter cylinder $\mathbb{S}^{d-1} \times \mathbb{R}$ via antipodal symmetrization exactly as in \eqref{eq:equation_nuw_f} yields:
\begin{equation*}
    \langle \mathcal{R}_\sigma^* \mu, \phi \rangle_{\mathbb{R}^d} = \frac{1}{(2\pi)^d} \int_{\mathbb{R}^d} \widehat{f}(\xi) \overline{\widehat{\phi}(\xi)} \, d\xi = \langle f, \phi \rangle_{\mathbb{R}^d}.
\end{equation*}
Since this holds unconditionally for any $\phi \in \mathcal{S}(\mathbb{R}^d)$, we conclude that $\mathcal{R}_\sigma^* \mu = f$ in the sense of distributions, completing the proof.
\end{proof}

We now establish the existence of an optimal linear isometry between the generalized Barron space $\mathcal{B}$ and $\mathcal{M}_p(\mathbb{S}^{d-1} \times \mathbb{R})$ for  $p \ge k_\sigma + \alpha + 2$.
\begin{theorem}[Existence of the Optimal Isometric Linear Operator]\label{th:optimal_operator}
Assume $p \ge k_\sigma + \alpha + 2$. For any activation $\sigma$ satisfying Hypothesis \ref{hyp:non_zero_g} with an even regularized spectrum $\widehat{g}$ (i.e., $\widehat{g}(-\rho) = \widehat{g}(\rho)$), there strictly exists an isometric linear operator $P : \mathcal{B} \to \mathcal{M}_p(\mathbb{S}^{d-1} \times \mathbb{R})$ that acts as an optimal right inverse for $\mathcal{R}_\sigma^*$. This operator assigns to each function $f \in \mathcal{B}$ a valid representing measure that exactly achieves the infimum of the Barron norm:
\begin{equation}
    \|Pf\|_{\mathrm{TV}, p} = \|f\|_{\mathcal{B}}.
\end{equation}
\end{theorem}

\begin{proof}
Let $f \in \mathcal{B}$ be a target function. By Definition \ref{def:generalized_barron_space}, $f \in L^1_p(\mathbb{R}^d)$ and there exists at least one representing measure $\mu_0 = (\mu_0)_w \underline{dw} \in \mathcal{M}_p(\mathbb{S}^{d-1} \times \mathbb{R})$ such that $\mathcal{R}_\sigma^* \mu_0 = f$. The Barron quotient norm is given by the infimum over all such representing measures:
\begin{equation*}
    \|f\|_{\mathcal{B}} = \inf \left\{ \|\mu\|_{\mathrm{TV}, p} \mid \mu \in \mathcal{M}_{p}(\mathbb{S}^{d-1} \times \mathbb{R}), \ \mu = \mu_w \underline{dw}, \ \mathcal{R}_\sigma^* \mu = f \right\}.
\end{equation*}

Let $\epsilon_\sigma = (-1)^\alpha \in \{-1, 1\}$ denote the parity factor. We define the central spatial inversion operator on measures $I : \mathcal{M}_p \to \mathcal{M}_p$ by $I(\mu)(w, b) = \mu(-w, -b)$. Because the bias weight function $W_p(b) = (1 + |b|)^p$ is perfectly symmetric ($W_p(b) = W_p(-b)$), the operator $I$ acts as an exact isometry, meaning $\|I(\mu)\|_{\mathrm{TV}, p} = \|\mu\|_{\mathrm{TV}, p}$. Furthermore, since the spherical Lebesgue measure $\underline{dw}$ is invariant under the antipodal mapping $w \mapsto -w$, the operator $I$ strictly preserves the absolute continuity of the measure with respect to $\underline{dw}$. In the frequency domain, the 1D fiber Fourier transform of the inverted measure evaluates to $\mathcal{F}_b\{I(\mu)_w\}(\rho) = \widehat{\mu}_{-w}(-\rho)$.

Let $\mu_0 = (\mu_0)_w \underline{dw} \in \mathcal{M}_p(\mathbb{S}^{d-1} \times \mathbb{R})$ be an arbitrary valid absolutely continuous representing measure for $f$. Since  $f \in L^1_p(\mathbb{R}^d)$ with $p \ge k_\sigma + \alpha + 2$, Lemma \ref{lem:spectral_characterization_f} applies directly. Its representation of $f$ implies that it satisfies the spectral condition:
\begin{equation*}
    \widehat{(\mu_0)}_w(\rho) + \epsilon_\sigma \widehat{(\mu_0)}_{-w}(-\rho) = S_f(w, \rho),
\end{equation*}
where $S_f(w, \rho) := (i\rho)^\alpha \frac{|\mathbb{S}^{d-1}|}{(2\pi)^{d-1}} \frac{|\rho|^{d-1}}{\widehat{g}(\rho)} \widehat{f}(\rho w)$ is a source term entirely determined by the target function $f$.

We can define our candidate optimal measure $\mu_{\mathrm{opt}}$ by symmetrizing $\mu_0$ according to the parity of $\sigma$:
\begin{equation*}
    \mu_{\mathrm{opt}} := \frac{1}{2} \Big( \mu_0 + \epsilon_\sigma I(\mu_0) \Big).
\end{equation*}

\paragraph{Step 1: Uniqueness and Linearity.}
Let us compute the fiber spectrum of $\mu_{\mathrm{opt}}$. By linearity of the Fourier transform:
\begin{equation*}
    \widehat{(\mu_{\mathrm{opt}})}_w(\rho) = \frac{1}{2} \Big( \widehat{(\mu_0)}_w(\rho) + \epsilon_\sigma \widehat{(\mu_0)}_{-w}(-\rho) \Big) = \frac{1}{2} S_f(w, \rho).
\end{equation*}
This identity reveals a key property: the entire spectrum of $\mu_{\mathrm{opt}}$ is explicitly and uniquely defined by $S_f$, which relies solely on $f$. Consequently, the measure $\mu_{\mathrm{opt}}$ is absolutely independent of the initial choice of $\mu_0$. This guarantees that the mapping $P(f) := \mu_{\mathrm{opt}}$ is an unambiguously defined linear operator, isolating the unique minimal-norm representative satisfying the parity symmetry $I(\mu) = \epsilon_\sigma \mu$.

\paragraph{Step 2: Representability.}
We must verify that $\mu_{\mathrm{opt}}$ is still a valid representing measure for $f$. By injecting it into the left-hand side of the condition from Lemma \ref{lem:spectral_characterization_f}:
\begin{align*}
    \widehat{(\mu_{\mathrm{opt}})}_w(\rho) + \epsilon_\sigma \widehat{(\mu_{\mathrm{opt}})}_{-w}(-\rho) &= \mathcal{F}_b\Big\{ \mu_{\mathrm{opt}} + \epsilon_\sigma I(\mu_{\mathrm{opt}}) \Big\}_w(\rho) \\
    &= \mathcal{F}_b\left\{ \frac{1}{2}(\mu_0 + \epsilon_\sigma I(\mu_0)) + \frac{\epsilon_\sigma}{2}(I(\mu_0) + \epsilon_\sigma I(I(\mu_0))) \right\}_w(\rho).
\end{align*}
Since $I(I(\mu_0)) = \mu_0$ and $\epsilon_\sigma^2 = 1$, the right-hand side identically simplifies to:
\begin{equation*}
    \mathcal{F}_b\Big\{ \mu_0 + \epsilon_\sigma I(\mu_0) \Big\}_w(\rho) = \widehat{(\mu_0)}_w(\rho) + \epsilon_\sigma \widehat{(\mu_0)}_{-w}(-\rho) = S_f(w, \rho).
\end{equation*}
Because $\widehat{g}(0) \neq 0$ by Hypothesis \ref{hyp:non_zero_g}, the source term $S_f(w, \rho)$ possesses a root at the origin of exact order $\alpha + d - 1 \ge \alpha$. Thus, $\mu_{\mathrm{opt}}$ satisfies both the spectral equation \eqref{eq:spectral_kernel_cond} and the origin-vanishing moments condition \eqref{eq:origin_vanishing_moments}. Lemma \ref{lem:spectral_characterization_f} therefore strictly guarantees that $\mathcal{R}_\sigma^* \mu_{\mathrm{opt}} = f$.

\paragraph{Step 3: Optimality.}
Finally, applying the triangle inequality to our spatial definition of $\mu_{\mathrm{opt}}$ yields:
\begin{equation*}
    \|\mu_{\mathrm{opt}}\|_{\mathrm{TV}, p} = \left\| \frac{1}{2} (\mu_0 + \epsilon_\sigma I(\mu_0)) \right\|_{\mathrm{TV}, p} \le \frac{1}{2} \Big( \|\mu_0\|_{\mathrm{TV}, p} + \|I(\mu_0)\|_{\mathrm{TV}, p} \Big) = \|\mu_0\|_{\mathrm{TV}, p}.
\end{equation*}
This inequality proves that the optimal linear component $\mu_{\mathrm{opt}}$ is bounded by the norm of any initial representing measure $\mu_0$. Taking the infimum over all such valid measures $\mu_0$ gives:
\begin{equation*}
    \|\mu_{\mathrm{opt}}\|_{\mathrm{TV}, p} \le \inf_{\mu_0} \|\mu_0\|_{\mathrm{TV}, p} = \|f\|_{\mathcal{B}}.
\end{equation*}
However, since $\mu_{\mathrm{opt}}$ is itself a valid representing measure, its norm is inherently bounded from below by the infimum ($\|\mu_{\mathrm{opt}}\|_{\mathrm{TV}, p} \ge \|f\|_{\mathcal{B}}$).

We rigorously conclude that $\|\mu_{\mathrm{opt}}\|_{\mathrm{TV}, p} = \|f\|_{\mathcal{B}}$. The linear operator $P(f) = \mu_{\mathrm{opt}}$ therefore inherently isolates a norm-minimizing element (which is the unique minimizer obeying the parity symmetry $I(\mu) = \epsilon_\sigma \mu$), completing the proof.
\end{proof}

\begin{remark}[Scope of the Parity Condition and the ELU Activation]
While the general distributional representation framework developed in Sections \ref{sec:caracterization_barron_distrib} and \ref{sec:caracterization_barron_tv} applies to any activation satisfying Hypothesis \ref{hyp:complete_hyp_sigma}---including asymmetric activations such as the Exponential Linear Unit (ELU) whose regularized spectrum $\widehat{g}$ is rational and asymmetric---Theorem \ref{th:optimal_operator} strictly requires $\widehat{g}$ to be an even function ($\widehat{g}(-\rho) = \widehat{g}(\rho)$). 
This parity symmetry is essential to isolate the central inversion operator $I$ and ensure that the symmetrized measure $\mu_{\mathrm{opt}}$ satisfies the representation equation identically on the punctured frequency line. Consequently, for activations with non-symmetric regularized spectra, representing measures still exist in $\mathcal{M}_p(\mathbb{S}^{d-1} \times \mathbb{R})$, but their optimal minimal-norm selection remains governed by the general variational infimum rather than by the linear operator $P$.
\end{remark}

\appendix
\section{Fundamental Theorems of Distribution Theory}\label{sec:appendix_distributions}

This appendix recalls foundational results and establishes technical lemmas from the Schwartz theory of distributions that are used to justify the regularity, duality extensions, and the interchange of integrals throughout this article.

\begin{theorem}[Structure Theorem for Tempered Distributions]\label{th:structure_distrib}
Let $T \in \mathcal{S}'(\mathbb{R})$ be a one-dimensional tempered distribution. Then, there exist an integer $m \ge 0$, a constant $C > 0$, an integer $k \ge 0$, and a continuous function $F : \mathbb{R} \to \mathbb{R}$ such that:
\begin{equation}
    T = D^m F \quad \text{in } \mathcal{S}'(\mathbb{R}),
\end{equation}
where $F$ exhibits at most polynomial growth: $|F(x)| \le C(1 + |x|)^k$ for all $x \in \mathbb{R}$. More generally, any tempered distribution on $\mathbb{R}^n$ can be represented as a finite sum $T = \sum_{|\alpha| \le m} D^\alpha F_\alpha$ of distributional derivatives of continuous functions with polynomial growth.
\end{theorem}

\begin{theorem}[Integration of Distributions Depending on a Parameter]\label{th:inversion_integrale_crochet}
Let $(\Lambda, \mathcal{A}, \mu)$ be a measure space, $T \in \mathcal{S}'(\mathbb{R}^n)$ be a tempered distribution, and $u : \Lambda \times \mathbb{R}^n \to \mathbb{R}$ be a joint function depending on a parameter $\lambda \in \Lambda$ and a spatial variable $x \in \mathbb{R}^n$. Suppose that:
\begin{enumerate}
    \item For $\mu$-almost every $\lambda \in \Lambda$, the mapping $x \mapsto u(\lambda, x)$ is a test function in $\mathcal{S}(\mathbb{R}^n)$.
    \item The parameter mapping $\lambda \mapsto u(\lambda, \cdot)$ is Bochner-integrable from $\Lambda$ into the Fréchet space $\mathcal{S}(\mathbb{R}^n)$, \textit{i.e.}, $\lambda \mapsto u(\lambda, \cdot)$ is strongly measurable and for every pair of multi-indices $\alpha, \beta \in \mathbb{N}^n$, the integral of the corresponding Schwartz seminorm is finite:
    \begin{equation}
        \int_{\Lambda} \sup_{x \in \mathbb{R}^n} \left| x^\alpha \partial_x^\beta u(\lambda, x) \right| d\mu(\lambda) < \infty.
    \end{equation}
\end{enumerate}
Under these conditions, the integrated function defined by $v(x) = \int_{\Lambda} u(\lambda, x) \, d\mu(\lambda)$ belongs to $\mathcal{S}(\mathbb{R}^n)$, and the action of the distribution $T$ rigorously commutes with the integral over the parameter $\lambda$:
\begin{equation}
    \left\langle T_x, \int_{\Lambda} u(\lambda, x) \, d\mu(\lambda) \right\rangle_{\mathbb{R}^n} = \int_{\Lambda} \langle T_x, u(\lambda, x) \rangle_{\mathbb{R}^n} \, d\mu(\lambda).
\end{equation}
\end{theorem}

\begin{lemma}[Distribution Order and Fourier Duality for Polynomial Growth Functions]\label{lem:fourier_order_poly_growth}
Let $f \in C^0(\mathbb{R})$ be a continuous function with polynomial growth of order $k \ge 0$, meaning there exists a constant $C_f > 0$ such that $|f(x)| \le C_f (1 + |x|)^k$ for all $x \in \mathbb{R}$. Then:
\begin{enumerate}
    \item The distributional Fourier transform $\widehat{f} \in \mathcal{S}'(\mathbb{R})$ has distribution order at most $k + 2$. Specifically, there exists a constant $C' > 0$ such that for all $\psi \in \mathcal{S}(\mathbb{R})$:
    \begin{equation}\label{eq:bound_fourier_order_k2}
        |\langle \widehat{f}, \psi \rangle| \le C' \sum_{m=0}^{k+2} \|\psi^{(m)}\|_{L^1(\mathbb{R})}.
    \end{equation}
    \item Consequently, for any integer $p \ge k + 2$, the linear functional $\widehat{f}$ extends uniquely by continuity to the Fr\'echet space $\mathcal{S}^p(\mathbb{R})$ of $C^p(\mathbb{R})$ functions whose derivatives up to order $p$ decay faster than any polynomial. Furthermore, for every $\theta \in \mathcal{S}^p(\mathbb{R})$, the Fourier duality identity holds:
    \begin{equation}\label{eq:extended_fourier_duality}
        \int_{\mathbb{R}} f(x) \widehat{\theta}(x) \, dx = \langle \widehat{f}, \theta \rangle,
    \end{equation}
    where $\widehat{\theta}(x) := \int_{\mathbb{R}} \theta(\rho) e^{-i \rho x} \, d\rho$.
\end{enumerate}
\end{lemma}

\begin{proof}
Let $\psi \in \mathcal{S}(\mathbb{R})$ be an arbitrary Schwartz test function. By definition of the Fourier transform on tempered distributions, the action of $\widehat{f}$ on $\psi$ is given by:
\begin{equation*}
    \langle \widehat{f}, \psi \rangle = \int_{\mathbb{R}} f(x) \widehat{\psi}(x) \, dx.
\end{equation*}
Because $|f(x)| \le C_f (1+|x|)^k$, we introduce the integrable Cauchy factor $(1+x^2)^{-1} \in L^1(\mathbb{R})$ to control the polynomial growth of $f$:
\begin{align*}
    |\langle \widehat{f}, \psi \rangle| &\le C_f \int_{\mathbb{R}} (1+|x|)^k |\widehat{\psi}(x)| \, dx \\
    &= C_f \int_{\mathbb{R}} \frac{1}{1+x^2} (1+x^2)(1+|x|)^k |\widehat{\psi}(x)| \, dx \\
    &\le C_f \left( \int_{\mathbb{R}} \frac{dx}{1+x^2} \right) \sup_{x \in \mathbb{R}} \left( (1+|x|)^{k+2} |\widehat{\psi}(x)| \right) \\
    &= \pi C_f \sup_{x \in \mathbb{R}} \left( (1+|x|)^{k+2} |\widehat{\psi}(x)| \right).
\end{align*}
Using the standard inequality $(1+|x|)^{k+2} \le 2^{k+1} (1 + |x|^{k+2})$, we obtain:
\begin{equation*}
    |\langle \widehat{f}, \psi \rangle| \le 2^{k+1} \pi C_f \left( \sup_{x \in \mathbb{R}} |\widehat{\psi}(x)| + \sup_{x \in \mathbb{R}} |x^{k+2} \widehat{\psi}(x)| \right).
\end{equation*}
For any integer $0 \le m \le k+2$, since $\psi \in \mathcal{S}(\mathbb{R})$, successive integration by parts (with boundary terms vanishing at infinity) gives:
\begin{equation*}
    (ix)^m \widehat{\psi}(x) = (ix)^m \int_{\mathbb{R}} \psi(\rho) e^{-i\rho x} \, d\rho = \int_{\mathbb{R}} \psi^{(m)}(\rho) e^{-i\rho x} \, d\rho = \mathcal{F}\{\psi^{(m)}\}(x).
\end{equation*}
Taking the supremum over $x \in \mathbb{R}$ immediately yields:
\begin{equation*}
    \sup_{x \in \mathbb{R}} |x^m \widehat{\psi}(x)| \le \int_{\mathbb{R}} |\psi^{(m)}(\rho)| \, d\rho = \|\psi^{(m)}\|_{L^1(\mathbb{R})}.
\end{equation*}
Combining these bounds proves the estimate \eqref{eq:bound_fourier_order_k2} with $C' = 2^{k+1} \pi C_f$, establishing that $\widehat{f} \in \mathcal{S}'(\mathbb{R})$ has distribution order at most $k+2$.

To prove the second assertion, let $p \ge k+2$ and let $\theta \in \mathcal{S}^p(\mathbb{R})$. Since $\theta \in C^p(\mathbb{R})$ and all its derivatives up to order $p$ decay faster than any polynomial, integrating by parts $m$ times for any $0 \le m \le k+2$ yields $\sup_{x \in \mathbb{R}} |x^m \widehat{\theta}(x)| \le \|\theta^{(m)}\|_{L^1(\mathbb{R})} < \infty$. Hence, the integral $\int_{\mathbb{R}} f(x) \widehat{\theta}(x) \, dx$ is absolutely convergent and bounded by $C' \sum_{m=0}^{k+2} \|\theta^{(m)}\|_{L^1(\mathbb{R})}$.

Finally, let $(\theta_n)_{n \in \mathbb{N}} \subset \mathcal{S}(\mathbb{R})$ be a sequence of smooth test functions approximating $\theta$ in $\mathcal{S}^p(\mathbb{R})$ (obtained by standard mollification and truncation). For each $n$, the definition of $\widehat{f} \in \mathcal{S}'(\mathbb{R})$ yields $\int_{\mathbb{R}} f(x) \widehat{\theta}_n(x) \, dx = \langle \widehat{f}, \theta_n \rangle$. Since $\theta_n \to \theta$ in $\mathcal{S}^p(\mathbb{R})$, the continuous bound \eqref{eq:bound_fourier_order_k2} implies that $\int_{\mathbb{R}} f(x) \widehat{\theta}_n(x) \, dx \to \int_{\mathbb{R}} f(x) \widehat{\theta}(x) \, dx$ and $\langle \widehat{f}, \theta_n \rangle \to \langle \widehat{f}, \theta \rangle$, which establishes the identity \eqref{eq:extended_fourier_duality} and concludes the proof.
\end{proof}

\begin{theorem}[Associativity of Multiplication with Regular Multipliers]\label{th:associativity_distrib}
Let $T \in \mathcal{S}'(\mathbb{R}^n)$ be a tempered distribution of finite order $m \ge 0$. Let $P : \mathbb{R}^n \to \mathbb{C}$ be a polynomial function, and let $\Theta \in C^p(\mathbb{R}^n)$ with $p \ge m$ be a regular multiplier such that $\Theta$ and all its partial derivatives up to order $m$ have at most polynomial growth. Then, the sequential distributional products are well-defined in $\mathcal{S}'(\mathbb{R}^n)$ and satisfy the strict associativity identity:
\begin{equation}
    \Theta \cdot (P \cdot T) = (P \Theta) \cdot T = P \cdot (\Theta \cdot T) \quad \text{in } \mathcal{S}'(\mathbb{R}^n).
\end{equation}
\end{theorem}

\begin{proof}
Because $T \in \mathcal{S}'(\mathbb{R}^n)$ is a tempered distribution of finite order $m$, it is continuous with respect to the Schwartz seminorms involving partial derivatives of order at most $m$. Consequently, the duality pairing $\langle T, \psi \rangle$ extends uniquely as a continuous linear functional to the space $\mathcal{S}^m(\mathbb{R}^n)$ of $C^m$ functions whose derivatives up to order $m$ decay faster than the reciprocal of any polynomial.

Let $\phi \in \mathcal{S}(\mathbb{R}^n)$ be an arbitrary Schwartz test function. Under the stated polynomial growth and $C^p$ regularity assumptions ($p \ge m$), the products $\Theta \phi$, $P \phi$, and $(P \Theta) \phi$ are $C^m$ functions with rapid decay, meaning they strictly belong to $\mathcal{S}^m(\mathbb{R}^n)$. 

By the definition of multiplication via transposition extended to finite-order distributions, the action of a multiplier $h \in C^m(\mathbb{R}^n)$ on $T$ is well-defined by the pairing $\langle h \cdot T, \phi \rangle = \langle T, h \phi \rangle$. Applying this definition sequentially to the left-hand side yields:
\begin{equation*}
    \langle \Theta \cdot (P \cdot T), \phi \rangle_{\mathbb{R}^n} = \langle P \cdot T, \Theta \phi \rangle_{\mathbb{R}^n} = \langle T, P (\Theta \phi) \rangle_{\mathbb{R}^n}.
\end{equation*}
In the algebra of ordinary complex-valued functions on $\mathbb{R}^n$, multiplication is strictly associative and commutative pointwise, meaning $P(x) (\Theta(x) \phi(x)) = (P(x) \Theta(x)) \phi(x) = \Theta(x) (P(x) \phi(x))$ for all $x \in \mathbb{R}^n$. Substituting this pointwise algebraic identity back into the well-defined duality bracket gives:
\begin{equation*}
    \langle T, P (\Theta \phi) \rangle_{\mathbb{R}^n} = \langle T, (P \Theta) \phi \rangle_{\mathbb{R}^n} = \langle (P \Theta) \cdot T, \phi \rangle_{\mathbb{R}^n},
\end{equation*}
and similarly:
\begin{equation*}
    \langle T, \Theta (P \phi) \rangle_{\mathbb{R}^n} = \langle \Theta \cdot T, P \phi \rangle_{\mathbb{R}^n} = \langle P \cdot (\Theta \cdot T), \phi \rangle_{\mathbb{R}^n}.
\end{equation*}
Since the equality of brackets $\langle \Theta \cdot (P \cdot T), \phi \rangle_{\mathbb{R}^n} = \langle (P \Theta) \cdot T, \phi \rangle_{\mathbb{R}^n} = \langle P \cdot (\Theta \cdot T), \phi \rangle_{\mathbb{R}^n}$ holds unconditionally for any arbitrary test function $\phi \in \mathcal{S}(\mathbb{R}^n)$, the distributions are strictly identical in $\mathcal{S}'(\mathbb{R}^n)$.
\end{proof}

\begin{corollary}[Spectral Simplification via Regularizing Factors]\label{cor:spectral_simplification}
Let $T \in \mathcal{S}'(\mathbb{R})$ be a tempered distribution of finite order $m \ge 0$. Let $P(\rho) = (i\rho)^\alpha$ with $\alpha \in \mathbb{N}$ such that $P \cdot T = G \in C^\infty(\mathbb{R})$ with at most polynomial growth. Let $\Theta \in C^0(\mathbb{R}) \cap C^\infty(\mathbb{R} \setminus \{0\})$ have polynomial growth, and assume that the composite multiplier $\Psi(\rho) := P(\rho)\Theta(\rho)$ belongs to $C^m(\mathbb{R})$ with polynomial growth up to order $m$, and satisfies:
\begin{equation}
    \partial_\rho^k \Psi(0) = 0 \quad \text{for all } 0 \le k \le m.
\end{equation}
Then, the sequential distributional product $\Psi \cdot T$ is well-defined in $\mathcal{S}'(\mathbb{R})$ and strictly coincides with the regular tempered distribution defined by the continuous function $\Theta G$:
\begin{equation}
    \Psi \cdot T = \Theta G \quad \text{in } \mathcal{S}'(\mathbb{R}).
\end{equation}
In particular, for any test function $\phi \in \mathcal{S}(\mathbb{R})$:
\begin{equation}
    \langle \Psi \cdot T, \phi \rangle = \int_{\mathbb{R}} \Theta(\rho) G(\rho) \phi(\rho) \, d\rho.
\end{equation}
\end{corollary}

\begin{proof}
Let $\phi \in \mathcal{S}(\mathbb{R})$ be an arbitrary test function and set $\psi(\rho) := \Psi(\rho)\phi(\rho) \in \mathcal{S}^m(\mathbb{R})$. Since $\phi$ is smooth and $\partial_\rho^k \Psi(0) = 0$ for all $0 \le k \le m$, Leibniz's rule implies that $\psi^{(k)}(0) = 0$ for all $0 \le k \le m$.

Let $\chi \in C_c^\infty(\mathbb{R})$ be a smooth cutoff function such that $0 \le \chi \le 1$, $\chi(\rho) = 1$ for $|\rho| \le 1/2$, and $\chi(\rho) = 0$ for $|\rho| \ge 1$. For any $\epsilon > 0$, set $\chi_\epsilon(\rho) = \chi(\rho / \epsilon)$ and $\eta_\epsilon(\rho) = 1 - \chi_\epsilon(\rho)$. We decompose $\psi = \eta_\epsilon \psi + \chi_\epsilon \psi$.

On the support of $\eta_\epsilon$, $|\rho| \ge \epsilon/2 > 0$, so $P(\rho) \neq 0$. Since $\Theta \in C^\infty(\mathbb{R} \setminus \{0\})$, the function $\theta_\epsilon(\rho) := \frac{\eta_\epsilon(\rho)\psi(\rho)}{P(\rho)} = \eta_\epsilon(\rho) \Theta(\rho) \phi(\rho)$ belongs to $\mathcal{S}(\mathbb{R})$. By definition of the product $P \cdot T = G$, we have:
\begin{equation*}
    \langle T, \eta_\epsilon \psi \rangle = \langle T, P \theta_\epsilon \rangle = \langle P \cdot T, \theta_\epsilon \rangle = \int_{\mathbb{R}} G(\rho) \eta_\epsilon(\rho) \Theta(\rho) \phi(\rho) \, d\rho.
\end{equation*}
As $\epsilon \to 0$, $\eta_\epsilon(\rho) \to 1$ pointwise on $\mathbb{R} \setminus \{0\}$. Since $\Theta G \phi \in L^1(\mathbb{R})$, Lebesgue's Dominated Convergence Theorem ensures:
\begin{equation*}
    \lim_{\epsilon \to 0} \langle T, \eta_\epsilon \psi \rangle = \int_{\mathbb{R}} \Theta(\rho) G(\rho) \phi(\rho) \, d\rho.
\end{equation*}

It remains to show that $\lim_{\epsilon \to 0} \langle T, \chi_\epsilon \psi \rangle = 0$. Since $T$ is of finite order $m$, there exists $C > 0$ such that $|\langle T, f \rangle| \le C \sum_{k=0}^m \sup_{|\rho| \le \epsilon} |f^{(k)}(\rho)|$ for all $f \in C_c^m([-\epsilon, \epsilon])$. By Leibniz's rule:
\begin{equation*}
    (\chi_\epsilon \psi)^{(k)}(\rho) = \sum_{j=0}^k \binom{k}{j} \epsilon^{-(k-j)} \chi^{(k-j)}\left(\frac{\rho}{\epsilon}\right) \psi^{(j)}(\rho).
\end{equation*}
Because $\psi^{(l)}(0) = 0$ for all $0 \le l \le m$, Taylor's theorem with integral remainder implies that for each $0 \le j \le k \le m$ and $|\rho| \le \epsilon$:
\begin{equation*}
    |\psi^{(j)}(\rho)| \le \frac{|\rho|^{m-j}}{(m-j)!} \sup_{|t| \le \epsilon} |\psi^{(m)}(t) - \psi^{(m)}(0)| = \frac{|\rho|^{m-j}}{(m-j)!} \omega(\epsilon),
\end{equation*}
where $\omega(\epsilon) \to 0$ as $\epsilon \to 0$ by the continuity of $\psi^{(m)}$ at $0$. Therefore, on $[-\epsilon, \epsilon]$:
\begin{equation*}
    \epsilon^{-(k-j)} |\psi^{(j)}(\rho)| \le \epsilon^{-(k-j)} \frac{\epsilon^{m-j}}{(m-j)!} \omega(\epsilon) \le \epsilon^{m-k} \omega(\epsilon) \le \omega(\epsilon) \xrightarrow[\epsilon \to 0]{} 0.
\end{equation*}
Consequently, $\lim_{\epsilon \to 0} \sup_{|\rho| \le \epsilon} |(\chi_\epsilon \psi)^{(k)}(\rho)| = 0$ for each $0 \le k \le m$, which guarantees that $\lim_{\epsilon \to 0} \langle T, \chi_\epsilon \psi \rangle = 0$.

Combining the two limits yields $\langle \Psi \cdot T, \phi \rangle = \lim_{\epsilon \to 0} \langle T, \eta_\epsilon \psi + \chi_\epsilon \psi \rangle = \int_{\mathbb{R}} \Theta(\rho) G(\rho) \phi(\rho) \, d\rho$, completing the proof.
\end{proof}
\section*{Fundings}

This work is supported by Agence Nationale de Recherche (ANR), grant number ANR-24-CPJ1-0081-01 (BIOMOD). The author has no relevant financial or non-financial interests to disclose.

\section*{Declaration of generative AI and AI-assisted technologies in the manuscript preparation process}
During the preparation of this work the author used Gemini in order to proofread the manuscript, refine the structural clarity, and ensure compliance with the journal's formatting guidelines. After using this tool/service, the author reviewed and edited the content as needed and takes full responsibility for the content of the published article.

\bibliographystyle{plain} 

\bibliography{biblio} 

\end{document}